%% file: Low-201309-OPFTutorial-1.tex
 \documentclass[12pt, onecolumn]{IEEEtran}
\usepackage{mathptmx}       
\usepackage{helvet}         
\usepackage{courier}        
\usepackage{type1cm}        
\usepackage{makeidx}         
\usepackage{graphicx}        
\usepackage{multicol}        
\usepackage[bottom]{footmisc}
\usepackage{ifthen}
\usepackage{subfigure}
\usepackage[usenames,dvipsnames]{color}

\usepackage{setspace}
\makeindex             

\usepackage{graphics} 
\usepackage{graphicx} 
\usepackage{epsfig} 
\usepackage{mathptmx} 
\usepackage{times} 
\usepackage{amsmath} 
\usepackage{amssymb}  

\usepackage{amsfonts}

\usepackage{subfig}
\usepackage{verbatim}

\usepackage{comment}

\DeclareGraphicsExtensions{.pdf,.png,.jpg,.eps}

\def\ba{\begin{array}}
\def\ea{\end{array}}
\newcommand{\beq}{\begin{equation}}
\newcommand{\eeq}{\end{equation}}
\newcommand{\bq}{\begin{eqnarray}}
\newcommand{\eq}{\end{eqnarray}}
\newcommand{\bqn}{\begin{eqnarray*}}
\newcommand{\eqn}{\end{eqnarray*}}
\newcommand{\bee}{\begin{enumerate}}
\newcommand{\eee}{\end{enumerate}}
\newcommand{\bi}{\begin{itemize}}
\newcommand{\ei}{\end{itemize}}

\newcommand{\ii}{\textbf{i}}

\newboolean{showcomments}
\setboolean{showcomments}{true}
\newcommand{\slow}[1]{\ifthenelse{\boolean{showcomments}}
{ \textcolor{red}{(Steven says:  #1)}}{}}

\newtheorem{theorem}{Theorem}

\newtheorem{lemma}[theorem]{Lemma}
\newtheorem{corollary}[theorem]{Corollary}
\newtheorem{remark}{Remark}

\begin{document}

\markboth{IEEE Trans. on Control of Network Systems, 1(1):15--27, March 2014
(with proofs)}
{IEEE Trans. on Control of Network Systems, 1(1):15--27, March 2014
(with proofs)}

\title{Convex Relaxation of Optimal Power Flow\\
{\huge Part I: Formulations and Equivalence}
\thanks{
\hspace{-0.34in}
\textbf{Citation:
\emph{IEEE Transactions on Control of Network Systems, 15(1): 15--27, 
March 2014.}
}
This is an extended version with Appendices VIII and IX that provide some 
mathematical preliminaries and proofs of the main results.
All proofs can be found in their original papers.  We provide proofs here because
(i) it is convenient to have all proofs in one place and in a uniform notation, 
and (ii) some of the formulations and presentations here are slightly different
from those in the original papers.
\newline
A preliminary and abridged version has appeared in 
Proceedings of the IREP Symposium - Bulk Power System Dynamics and Control - IX, Rethymnon, Greece, August 25-30, 2013. 
}
}
\author{Steven H. Low \\ Electrical Engineering, Computing+Mathematical Sciences \\
	Engineering and Applied Science, Caltech \\ 
	slow@caltech.edu \\
}
\date{April 15, 2014}
\maketitle

\vspace{-0.5in}
\begin{center}
April 15, 2014
\end{center}
\vspace{0.5in}

\begin{abstract}
This tutorial summarizes recent advances in the convex relaxation of the optimal
power flow (OPF) problem, focusing on structural properties rather than algorithms.
 Part I presents two power flow models, formulates OPF and their relaxations in each 
 model, and  proves equivalence relations among them.
 Part II presents sufficient conditions under which the convex relaxations are
 exact. 
\end{abstract}

\newpage
 \tableofcontents
 \vfill{
\noindent
\textbf{Acknowledgment.}
We thank the support of NSF through NetSE CNS 0911041, 
ARPA-E through GENI DE-AR0000226, 
Southern California Edison, 
the National Science Council of Taiwan through
NSC 103-3113-P-008-001,
the Los Alamos National Lab (DoE),
and Caltech's Resnick Institute.
}
 \newpage

\input{1-intro}

\input{1-models-v2}
\input{1-opf-v2}

\input{1-relaxationsBIM}

\input{1-relaxationsBFM}

\input{1-BFMt}

\input{1-conc}

\newpage
\input{1-prelim}

\input{1-AppendixProofs}

\newpage
\bibliographystyle{unsrt}
\bibliography{PowerRef-201202}

\end{document}

%% file: 1-intro.tex
\section{Introduction}
\label{sec:intro1}

For our purposes an optimal power flow (OPF) problem is a mathematical program that
seeks to minimize a certain function, such as total power loss, generation cost 
or user disutility, subject to the Kirchhoff's laws as well as capacity, 
stability and security constraints.  
OPF is fundamental in power system operations as it underlies
many applications such as economic dispatch, unit commitment, state estimation,
stability and reliability assessment, volt/var control, demand response, etc.
There has been a great deal of research 
on OPF since Carpentier's first formulation in 1962 \cite{Carpentier62}.
An early solution appears in \cite{Dommel1968} and extensive surveys can be found in
e.g. \cite{Powerbk, Huneault91,Momoh99a,Momoh99b,Pandya08, Frank2012a, Frank2012b,
OPF-FERC-1, OPF-FERC-2, OPF-FERC-3, OPF-FERC-4, OPF-FERC-5}.

Power flow equations are quadratic and hence OPF can be formulated
as a quadratically constrained quadratic program (QCQP).   It is generally nonconvex and hence
NP-hard.   A large number of optimization algorithms and relaxations have been proposed.
A popular approximation is a linear program, called DC OPF, obtained through the
 linearization of the power flow equations
e.g. \cite{Stott1974,Alsac1990, Overbye2004, Purchala2005,Stott2009}.
See also \cite{Coffrin2012} for a more accurate linear approximation.
To the best of our knowledge solving OPF through semidefinite relaxation is first proposed in 
\cite{Jabr2006} as a second-order
cone program (SOCP) for radial (tree) networks and in \cite{Bai2008} as a semidefinite program 
(SDP) for general networks in a bus injection model.  It is first proposed 
in \cite{Farivar2011-VAR-SGC, Farivar-2013-BFM-TPS} as an SOCP for radial 
networks in the branch flow model of \cite{Baran1989a, Baran1989b}.
See Remark \ref{remark:bioCoReBIM} below for more details.
While these convex relaxations have been illustrated numerically in \cite{Jabr2006} and
\cite{Bai2008}, whether or when they will turn out to be exact is first studied in \cite{Lavaei2012}.
Exploiting graph sparsity to simplify the SDP relaxation of OPF is first proposed in 
\cite{Bai2011, Jabr2012} and analyzed in \cite{MolzahnLesieutre2013, Bose-2014-BFMe-TAC}.

Convex relaxation of quadratic programs has been applied to many engineering problems; 
see e.g. \cite{luo10}.  There is a rich theory and extensive empirical experiences.   
Compared with other approaches, solving OPF through convex relaxation offers several
\emph{advantages}.  First, while DC OPF is useful in a wide variety of applications, it is not
applicable in other applications;
see Remark \ref{remark:LDFvsDC}.
%
Second a solution of DC OPF may not be feasible (may not satisfy the nonlinear power flow equations).
In this case an operator may tighten some constraints in DC OPF and solve again.  This may not only
reduce efficiency but also relies on heuristics that are hard to scale to larger systems or faster control
in the future.
Third, when they converge, most nonlinear algorithms compute a local optimal usually without  
assurance on the quality of the solution.   In contrast a convex relaxation provides for the first
time the ability to check if a solution is globally optimal.   If it is not, the solution provides a lower
bound on the minimum cost and hence a bound on how far any feasible solution is from optimality.
Unlike approximations, if a relaxed problem is infeasible, it is a certificate that the original OPF is
infeasible.

This two-part tutorial explains the main theoretical results on semidefinite relaxations of OPF developed in the
last few years.
Part I presents two power flow models that are useful in different situations,
formulates OPF and its convex relaxations in each model, and clarifies their relationship.
Part II \cite{Low2014b}
presents sufficient conditions that guarantee the relaxations are exact, i.e. when one can
recover a globally optimal solution of OPF from an optimal solution of its relaxations.
We focus on basic results using the simplest OPF formulation and does not cover 
many relevant works in the literature, such as stochastic OPF e.g. 
\cite{Varaiya-2011-RiskLimiting-PIEEE, Vrakopoulou2013, BentBienstock2013},
distributed OPF e.g. \cite{KimBaldick1997, Baldick1999, HugAndersson2009, Nogales2003, lam2012distributed, KraningBoyd2013}, 
new applications e.g. \cite{Turitsyn11,  LamZhang2012},
or what to do when relaxation fails e.g. \cite{Phan2012, Gopalakrishnan2012, JoszPanciatici2013},
to name just a few.


\subsection{Outline of paper}

Many mathematical models have been used to model power networks.  
In Part I of this two-part paper we present two such models, we call the bus injection model (BIM) and the branch
flow model (BFM).  Each model consists of a set of power flow equations.
Each models a power network in that the solutions of each set of equations, called the power flow solutions,
describe the steady state of the network.
We prove that these two models are equivalent in the sense that there is 
a bijection between their solution sets (Section \ref{sec:models}).
We formulate OPF within each model where the power flow solutions define the feasible set of OPF
(Section \ref{sec:opf}).
Even though BIM and BFM are equivalent some results are much easier to formulate or prove
in one model than the other; see Remark \ref{remark:BIMeqBFM} in Section \ref{sec:models}.

The complexity of OPF formulated here lies in the nonconvexity of power flow equations that 
gives rise to a nonconvex feasible set of OPF.   We develop various characterizations of the
feasible set and design convex supersets based on these characterizations.  
Different designs lead to different convex relaxations and we prove their relationship 
(Sections \ref{sec:rBIM} and \ref{sec:rBFM}).
When a relaxation is exact an optimal solution of
the original nonconvex OPF can be recovered from \emph{any} optimal solution of the
relaxation. 
 In Part II \cite{Low2014b} 
we present sufficient conditions that guarantee the exactness of convex relaxations.

%
%

Branch flow models are originally proposed for networks with a tree topology, called \emph{radial
networks}, e.g. \cite{Baran1989a, Baran1989b, ChiangBaran1990, Chiang1991,
Cespedes1990, Exposito1999, Farivar2011-VAR-SGC, Taylor2012-TPS, LavaeiTseZhang2013}.
They take a recursive structure that simplifies the computation of power flow solutions,
e.g. \cite{Kersting2002, Shirmohammadi1988, ChiangBaran1990}.   The model
of \cite{Baran1989a, Baran1989b} also has a linearization that offers several 
advantages over DC OPF in BIM; see Remark \ref{remark:LDFvsDC}.
The linear approximation provides simple bounds on the branch powers and voltage magnitudes in the nonlinear BFM (Section \ref{sec:BFMt}).
These bounds are used in \cite{Gan-2014-BFMt-TAC} to prove a sufficient condition for exact relaxation.

We make algorithmic recommendations in Section \ref{sec:conc1} based
on the results presented here.

This extended version differs from the journal version only in the addition of two Appendices.
Appendix VIII provides some mathematical preliminaries and Appendix IX
proofs of all main results.  
Even though all proofs can be found in their original papers,
we provide proofs here because (i) it is convenient to have all proofs in one place
and in a uniform notation, 
and (ii) some of the formulations and presentations here are slightly different
from those in the original papers.

\subsection{Notations}

Let $\mathbb{C}$ denote the set of complex numbers, $\mathbb{R}$ the set of real numbers,
and $\mathbb N$ the set of integers.
For $a\in \mathbb C$, Re$\,\, a$ and Im$\,\, a$ denote the real and imaginary parts of $a$ 
respectively.
For any set $A \subseteq \mathbb C^n$,  conv$\, A$ denotes the convex hull of $A$.
For $a\in \mathbb R$, $[a]^+ := \max\{ a, 0 \}$.
For $a, b \in \mathbb C$, $a\leq b$ means Re$\,\, a \leq \,\,\,$Re$\,\, b$ and 
Im$\,\, a \leq \,\,\,$Im$\,\, b$.
We abuse notation to use the same symbol $a$ to denote either a complex number
Re$\, a + \ii \, $Im$\, a$ or a 2-dimensional real vector $a = $(Re$\, a$, Im$\, a$)
depending on the context.

In general scalar or vector variables are in small letters, e.g. $u, w, x, y, z$.  
Most power system quantities however are in capital letters, 
e.g. $S_{jk}, P_{jk}, Q_{jk}, 	I_j, V_j$.
A variable without a subscript  denotes a vector with appropriate components,
e.g. $s := (s_j, j=0, \dots, n)$, $S := (S_{jk}, (j,k)\in E)$.   
For vectors $x, y$, $x\leq y$ denotes componentwise inequality.  

Matrices are usually in capital letters.   The transpose of a matrix $A$ is denoted by $A^T$
 and its Hermitian (complex conjugate) transpose by $A^H$.  A matrix $A$ is Hermitian
if $A = A^H$.  $A$ is positive semidefinite (or psd), denoted by $A \succeq 0$,
if  $A$ is Hermitian \emph{and} $x^H A x \geq 0$ for all $x\in \mathbb C^n$;
in particular if $A\succeq 0$ then by definition $A = A^H$.
For matrices $A, B$, $A \succeq B$ means $A-B$ is psd. 
Let $\mathbb S^n$ be the set of all $n\times n$ Hermitian matrices and
$\mathbb S^n_+$ the set of $n\times n$ psd matrices.

A graph $G = (N, E)$ consists of a set $N$ of nodes and a set
$E \subseteq N\times N$ of edges.   If $G$ is undirected then $(j, k)\in E$ if and
only if $(k,j)\in E$.   If $G$ is directed then $(j, k)\in E$ only if $(k,j)\not\in E$;
in this case we will use $(j,k)$ and $j\rightarrow k$ interchangeably to denote
an edge pointing from $j$ to $k$. 
We sometimes use $\tilde G = (N, \tilde E)$ to denote a directed graph.
By ``$j\sim k$'' we mean an edge $(j,k)$ if $G$ is undirected and
either $j \rightarrow k$ or $k\rightarrow j$ if $G$ is directed.
Sometimes we write $j\in G$ or $(j,k)\in G$ to mean $j\in N$ or $(j,k)\in E$ 
respectively.
A \emph{cycle} $c := (j_1, \dots, j_K)$ is an ordered set of nodes $j_k\in N$
so that $(j_k, j_{k+1}) \in E$ for $k=1, \dots, K$ with the understanding that
$j_{K+1} := j_1$.  In that case we refer to a link or a node in the cycle by 
$(j_k, j_{k+1})\in c$ or $j_k\in c$ respectively.

%% file: 1-models-v2.tex
\section{Power flow models}
\label{sec:models}

In this section we describe two mathematical models of power networks and prove
their equivalence.
By a ``mathematical model'' we mean a set of variables and a set of equations relating
these variables.  These equations are motivated by the physical system, but mathematically,
they are the starting point from which all claims are derived.

\subsection{Bus injection model}

Consider a power network  modeled by a connected undirected graph 
$G(N^+, E)$ where $N^+ := \{0\} \cup N$, $N := \{1, 2, \ldots, n \}$,
and $E \subseteq N^+ \times N^+$.   Each node in $N^+$ represents a bus 
and each edge in $E$ represents a transmission or distribution line.  
We use ``bus'' and ``node'' interchangeably and ``line'' and ``edge'' interchangeably. 
For each edge $(i,j)\in E$
let $y_{ij} \in \mathbb C$ be its admittance.
A bus $j \in N^+$ can have a generator, a load, both or neither.  
Let $V_j$ be the complex voltage at bus $j \in N^+$ and $|V_j|$ denote its magnitude.
 Bus 0 is the \emph{slack bus}.  Its voltage is fixed and we assume without loss of generality
 that $V_0 = 1\angle 0^\circ$ per unit (pu).
Let $s_j$ be the net complex power injection
(generation minus load) at bus $j\in N^+$.

The \emph{bus injection model} (BIM) is defined by the following power flow equations
that describe the Kirchhoff's laws: 
\bq
s_j  & = & \sum_{k: j\sim k} y_{jk}^H\ V_j (V_j^H - V_k^H),
\ \ \  j \in N^+
\label{eq:bim.1}
\eq
Let the set of power flow solutions $V$ for each $s$ be:
\bqn
\mathbb{V}(s) & :=  & \{ V \in \mathbb{C}^{n+1} \ \vert \ V \text{ satisfies } \eqref{eq:bim.1} \}
\eqn
For convenience we include $V_0$ in the vector variable $V := (V_j, j\in N^+)$ with the
understanding that $V_0 := 1 \angle 0^\circ$ is fixed.

\begin{remark}
\label{remark:1}
\emph{Bus types.}
Each bus $j$ is characterized by two complex variables $V_j$ and $s_j$, or
equivalently, four real
variables.   The buses are usually classified into three types, depending on which two
of the four real variables are specified.  For the slack bus 0, $V_0$
is given and $s_0$ is variable.   For a \emph{generator bus} (also called $PV$-bus),
Re$(s_j) = p_j$ and $|V_j|$ are specified and Im$(s_j) = q_j$ and $\angle V_j$ are
variable.   For a \emph{load bus} (also called $PQ$-bus), $s_j$ is specified and $V_j$
is variable.
The \emph{power flow} or \emph{load flow} problem is: given two of the four real variables
specified for each bus, solve the $n+1$ \emph{complex} equations in \eqref{eq:bim.1} for 
the remaining $2(n+1)$ real variables.  
For instance when all $n$ buses $j\neq 0$ are all load buses,
the power flow problem solves \eqref{eq:bim.1} for the $n$ complex voltages $V_j, j\neq 0$, 
and the power injection $s_0$ at the slack bus 0.
This can model a distribution system with a substation at bus 0 and $n$ 
constant-power loads at the other buses.
For \emph{optimal power flow} problems $p_j$ and $|V_j|$ on generator buses or 
$s_j$ on load buses can be variables as well.
For instance economic dispatch optimizes real power generations $p_j$
at generator buses; demand response optimizes demands $s_j$ at load buses; and
volt/var control optimizes reactive powers $q_j$ at capacitor banks, tap changers, or inverters.
These remarks also apply to the branch flow model presented next.
\end{remark}

\subsection{Branch flow model}
\label{sec:bfm}

In the branch flow model we adopt a connected directed graph 
$\tilde G = (N^+, \tilde E)$ where each
node in $N^+ := \{0, 1, \dots, n\}$ represents a bus and each edge in
$\tilde E \subseteq N^+ \times N^+$ represents a
transmission or distribution line.   Fix an arbitrary orientation for $\tilde G$ and let 
 $m := |\tilde E|$ be the number of directed edges in $\tilde G$.
%
Denote an edge by $(j, k)$ or $j\rightarrow k$ if it points from node $j$ to node $k$.    
For each edge $(j,k)\in \tilde E$ let $z_{jk}:= 1/y_{jk}$ be the complex impedance on the line;
 let $I_{jk}$ be the complex current and 
$S_{jk} = P_{jk} + \ii Q_{jk}$ be the {\em sending-end} complex power from buses $j$ to $k$.
For each bus $j\in N^+$ let $V_j$ be the complex voltage at bus $j$.
Assume without loss of generality that $V_0 = 1\angle 0^\circ$ pu.
Let $s_j$ be the net complex power injection at bus $j$.

The \emph{branch flow model} (BFM)
in \cite{Farivar-2013-BFM-TPS} is defined by the following set of power flow equations:
\begin{subequations}
\bq
\!\!\!\!
\sum_{k: j\rightarrow k} \!\!\! S_{jk}  & \!\! = \!\! &\!\!\!\!   
	\sum_{i: i\rightarrow j} \!\! \left( S_{ij} - z_{ij} |I_{ij}|^2 \right) + s_j, 	\,  j \in N^+
\label{eq:bfm.3}
\\
\!\!\!\!
I_{jk}  & \!\! = \!\! &  y_{jk} (V_j - V_k),  \quad j \rightarrow k \in \tilde{E}
\label{eq:bfm.1}
\\
\!\!\!\!
S_{jk}  & \!\! = \!\! & V_j \, I_{jk}^H, \quad j \rightarrow k \in \tilde{E}
\label{eq:bfm.2}
\eq
\label{eq:bfm}
\end{subequations}
where \eqref{eq:bfm.1} is the Ohm's law, \eqref{eq:bfm.2} defines branch power, 
and \eqref{eq:bfm.3} imposes power balance at each bus.  The quantity
$z_{ij} |I_{ij}|^2$ represents line loss so that $S_{ij} - z_{ij} |I_{ij}|^2$ is the
receiving-end complex power at bus $j$ from bus $i$.

Let the set of solutions $\tilde{x} := (S, I, V)$ of BFM for each $s$ be:
\bqn
\mathbb{\tilde{X}}(s) & := & \{ \tilde{x} \in \mathbb{C}^{2m+n+1} \ \vert \ \tilde{x} \text{ satisfies } 
				\eqref{eq:bfm} \}
\eqn
For convenience we include $V_0$ in the vector variable $V := (V_j, j\in N^+)$ with the
understanding that $V_0 := 1 \angle 0^\circ$ is fixed.

\subsection{Equivalence}

Even though the bus injection model \eqref{eq:bim.1} and the branch flow model \eqref{eq:bfm}
are defined by different
sets of equations in terms of their own variables, both are models of the Kirchhoff's laws
and therefore must be related.  We now clarify the precise sense in which these two
mathematical models are equivalent.  
We say two sets $A$ and $B$ are \emph{equivalent}, denoted  by $A\equiv B$, if
there is a bijection between them \cite{Bose-2012-BFMe-Allerton}.
\begin{theorem}
\label{thm:bim=bfm}
$\mathbb V(s) \equiv \mathbb{\tilde X} (s)$ for any power injections $s$.
\end{theorem}
\vspace{0.1in}

\begin{remark}
\label{remark:BIMeqBFM}
\emph{Two models.} 
Given the bijection between the solution sets $\mathbb V(s)$ and $\mathbb{\tilde X}(s)$
any result in one model is in principle derivable in the other.   
Some results however are much easier to state or derive in one model than the other.
For instance BIM, which is widely used in transmission network problems, allows a
much cleaner formulation of the semidefinite program
(SDP) relaxation. 
BFM for radial networks has a convenient recursive structure that allows a more efficient 
computation of power flows and leads to a useful linear approximation of BFM; see Section \ref{sec:BFMt}.
The sufficient condition for exact relaxation in \cite{Gan-2014-BFMt-TAC} provides intricate insights
on power flows that are hard to formulate or prove in BIM.
Finally, since BFM directly models branch flows $S_{jk}$ and currents $I_{jk}$, it is  easier
to use for some applications.  We will therefore freely use either model depending on
which is more convenient for the problem at hand.  
\end{remark}

%% file: 1-opf-v2.tex
\section{Optimal power flow}
\label{sec:opf}

\subsection{Bus injection model}

As mentioned in Remark \ref{remark:1} an optimal power flow problem optimizes both variables $V$
and $s$ over the solution set of the BIM \eqref{eq:bim.1}.  In addition
all voltage magnitudes must satisfy:
 \bq
\underline{v}_j \ \leq &\!\!  |V_j|^2 \! \!&  \leq\  \overline{v}_j, \ \ 
  		j \in N^+
  \label{eq:opfv}
\eq
\noindent
where $\underline{v}_j$ and $\overline{v}_j$ are given lower and upper bounds on voltage
magnitudes.   Throughout this paper we assume $\underline{v}_j>0$ to avoid triviality.
The power injections  are also constrained:
\bq
 \underline{s}_j  \leq s_j \leq \overline{s}_j, \ \ \  j \in N^+
 \label{eq:opfs}
\eq
where $\underline{s}_j$ and $\overline{s}_j$ are given bounds on the injections
at buses $j$.  
\begin{remark}
\emph{OPF constraints.}
If there is no bound on the load or on the generation at bus $j$
then $\underline{s}_j = -\infty - \ii \infty$ or $\overline{s}_j = \infty + \ii \infty$ respectively.
On the other hand \eqref{eq:opfs} also allows the case where $s_j$ is fixed (e.g. a
constant-power load), by setting $\underline s_j = \overline{s}_j$ to the specified value.
For the slack bus 0, unless otherwise specified, we always assume 
$\underline v_0 = \overline v_0 = 1$ and 
$\underline{s}_0 = -\infty - \ii \infty$, $\overline{s}_0 = \infty + \ii \infty$.
Therefore we sometimes replace $j\in N^+$ in \eqref{eq:opfv} and \eqref{eq:opfs} by $j\in N$.
\end{remark}

We can eliminate the variables $s_j$ from the OPF formulation by combining
\eqref{eq:bim.1} and \eqref{eq:opfs} into
\bq
\underline{s}_j  \ \ \leq  \sum_{k: (j, k) \in E} y_{jk}^H\, V_j (V_j^H - V_k^H)  \ \ 
\leq \ \ \overline{s}_j, 
\  j \in N^+
\label{eq:bimopf.1}
\eq
Then OPF in the bus injection model can be defined just in terms of the complex voltage
vector $V$.    Define
\begin{align}
\mathbb{V} := \{ V \in \mathbb{C}^{n+1} \ \vert \ V \text{ satisfies } \eqref{eq:opfv}, \eqref{eq:bimopf.1} \}
\label{eq:defV}
\end{align}
$\mathbb V$ is the feasible set of  optimal power flow problems in BIM.

Let the cost function be $C(V)$.  Typical costs include the cost of
generating real power at each generator bus or line loss over the network.
All these costs can be expressed as  functions of $V$.   Then the problem of interest is:
\\
\noindent
\textbf{OPF:}
\bq
\label{eq:OPFbim}
\underset{V}{\text{min}}   \ C(V) \ & \text{subject to} &   V \in \mathbb V
\eq
Since \eqref{eq:bimopf.1} is quadratic, $\mathbb{V}$ is generally a nonconvex set.
OPF is thus a nonconvex problem and NP-hard to solve in general.

\subsection{Branch flow model}

Denote the variables in the branch flow model \eqref{eq:bfm} by 
$\tilde{x} := (S, I, V, s) \in \mathbb{C}^{2(m+n+1)}$.  
We can also eliminate the variables $s_j$ as for the bus injection model 
by combining \eqref{eq:bfm.3} and \eqref{eq:opfs} 
but it will prove convenient to retain $s:= (s_j, j\in N^+)$ as part of the variables.
Define the feasible set in the branch flow model:
\begin{align}
\mathbb{\tilde X} := \{ \tilde{x} \in \mathbb{C}^{2(m+n+1)} \ \vert \ \tilde{x} \text{ satisfies } 
				\eqref{eq:bfm},  \eqref{eq:opfv}, \eqref{eq:opfs} \}
\label{eq:deftX}
\end{align}

Let the cost function in the branch flow model be $C(\tilde x)$.
Then the optimal power flow problem in the branch flow model is:
\\
\noindent
\textbf{OPF:}
\bq
\label{eq:OPFbfm}
\underset{\tilde x } {\text{min}}   \ \  C(\tilde x) & \text{ subject to } &  \tilde x \in \mathbb{\tilde X}
\eq
Since \eqref{eq:bfm} is quadratic, $\mathbb{X}$ is generally a nonconvex set.
OPF is thus a nonconvex problem and NP-hard to solve in general. 

\vspace{0.1in}
\begin{remark}
\emph{OPF equivalence.}
By Theorem \ref{thm:bim=bfm} there is a bijection between $\mathbb V$ and
$\mathbb{\tilde X}$.
Throughout this paper we assume that the cost functions in BIM and BFM
are equivalent under this bijection and we abuse notation to denote them by the
same symbol $C(\cdot)$.
Then OPF \eqref{eq:OPFbim} in BIM and \eqref{eq:OPFbfm} in BFM
are equivalent.   
\end{remark}

\begin{remark}
\emph{OPF variants.}
OPF as defined in \eqref{eq:OPFbim} and \eqref{eq:OPFbfm}  is a simplified version 
that ignores other important constraints such as line limits,  security constraints,
stability constraints, and chance constraints; see extensive surveys in \cite{Powerbk, Huneault91,Momoh99a,Momoh99b,Pandya08, Frank2012a, Frank2012b,
OPF-FERC-1, OPF-FERC-2, OPF-FERC-3, OPF-FERC-4, OPF-FERC-5, GanThomas2000, Capitanescua2011}
and a recent discussion in  \cite{Stott2012} on real-life OPF problems.
Some of these can be incorporated without any 
change to the results in this paper (e.g. see \cite{Farivar-2013-BFM-TPS, Bose-2012-QCQPt} for
models that include shunt elements and line limits).
Indeed a shunt element $y_j$ at bus $j$ can be easily included in BIM
by modifying \eqref{eq:bim.1}
into:
\bqn
s_j  & = & \sum_{k: j\sim k} y_{jk}^H\ V_j (V_j^H - V_k^H) + y_j^H |V_j|^2
\eqn
or included in BFM by modifying \eqref{eq:bfm.3} into:
\bqn
\!\!\!\!
\sum_{k: j\rightarrow k} \!\!\! S_{jk}  + y_j^H |V_j|^2 & \!\! = \!\! &\!\!\!\!   
	\sum_{i: i\rightarrow j} \!\! \left( S_{ij} - z_{ij} |I_{ij}|^2 \right) + s_j
\eqn
\end{remark}

\subsection{OPF as QCQP}
\label{subsec:opfqcqp}

Before we describe convex relaxations of OPF  we first show that,
when $C(V) := V^H C V$ is quadratic in $V$ for some Hermitian matrix $C$, OPF 
is indeed a quadratically constrained quadratic program (QCQP)
by converting it into the standard form.  
We will use the derivation in \cite{Bose-2012-QCQPt} for OPF \eqref{eq:OPFbim}
in BIM.   OPF \eqref{eq:OPFbfm} in BFM can similarly be converted into
a standard form QCQP.

Define the $(n+1) \times (n+1)$ admittance matrix $Y$ by
\bqn
Y_{ij} & = & \begin{cases}
\displaystyle
\sum_{k: k \sim i} y_{ik}, & \text{ if } i=j \\
- y_{ij}, & \text{ if } i \neq j \text { and $i \sim j$}  \\
0 & \text{ otherwise}  \end{cases}
\eqn
$Y$ is symmetric but not necessarily Hermitian.   
Let $I_j$ be the net injection current from bus $j$ to the rest of the network.
Then the current vector $I$ and the voltage vector $V$ are related by the
Ohm's law $I = YV$. BIM \eqref{eq:bim.1} is equivalent to:
\bqn
s_j & = & V_j I_j^H \ = \ (e_j^H V) (I^H e_j)
\eqn
where $e_j$ is the $(n+1)$-dimensional vector with 1 in the $j$th entry
and 0 elsewhere.  Hence, since $I = YV$, we have
\bqn
s_j \ = \ \text{tr } \left( e_j^H VV^H Y^H e_j \right) 
	& = & \text{tr } \left( Y^H e_j e_j^H \right) VV^H
	\ = \ V^H Y_j^H V
\eqn
where $Y_j := e_j e_j^H Y$ is an $(n+1) \times (n+1)$ matrix with its $j$th row
equal to the $j$th row of the admittance matrix $Y$ and all other rows equal
to the zero vector.   $Y_j$ is  in general not Hermitian so that 
$V^H Y_j^H V$ is in general a complex number.
Its real and imaginary parts can be expressed in terms of the Hermitian and
skew Hermitian components of $Y_j^H$ defined as:
\bqn
\!\!\!\!\!
\Phi_j := \frac{1}{2} \left( Y_j^H + Y_j \right) &\!\!\!\!\!  \text{and} \!\!\!\!\! & 
\Psi_j := \frac{1}{2\ii} \left( Y_j^H - Y_j \right)
\label{eq:defPhiPsi}
\eqn
Then 
\bqn
\text{Re}\,\, s_j = V^H \Phi_j V  & \text{and} & 
\text{Im}\, \, s_j = V^H \Psi_j V
\eqn
Let their upper and lower bounds be denoted by
\bqn
\underline{p}_j := \text{Re}\,\, \underline{s}_j & \text{and} & 
\overline{p}_j := \text{Re}\,\, \overline{s}_j
\\
\underline{q}_j := \text{Re}\,\, \underline{s}_j & \text{and} & 
\overline{q}_j := \text{Re}\,\, \overline{s}_j
\eqn
Let $J_j := e_j e_j^H$ denote the Hermitian matrix with a single 1 in
the $(j,j)$th entry and 0 everywhere else.
Then OPF \eqref{eq:OPFbim} can be written as a standard form QCQP:
\begin{subequations}
\bq
\!\!\!\!\!\!    \!\!\!\!\!\!\!
\min_{V\in \mathbb C^{n+1}} & \!\!\!\! & \!\!\!\!  V^H C V
\label{eq:bimOPF.b1}
\\
\!\!\!\!\!\!       \!\!\!\!\!\!\!
\text{subject to}  & \!\!\!\! & \!\!\!\!
 	V^H \Phi_j V \leq \overline{p}_j, \  	V^H (-\Phi_j) V \leq -\underline{p}_j
\label{eq:bimOPF.b2}
\\
\!\!\!\!\!\!       \!\!\!\!\!\!\!   & \!\!\!\!  & \!\!\!\!
 		V^H \Psi_j V \leq \overline{q}_j, \ V^H (-\Psi_j) V \leq -\underline{q}_j
\label{eq:bimOPF.b3}
\\
\!\!\!\!\!\!       \!\!\!\!\!\!\!   & \!\!\!\!  & \!\!\!\!
		V^H J_j V \leq \overline{v}_j, \ V^H (- J_j) V \leq -\underline{v}_j
\label{eq:bimOPF.b5}
\eq
\label{eq:OPFbim.2}
\end{subequations}
where $j\in N^+$ in \eqref{eq:OPFbim.2}.

%% file: 1-relaxationsBIM.tex
\section{Feasible sets and relaxations: BIM}
\label{sec:rBIM}

In this and the next section we derive semidefinite relaxations of OPF and clarify their relations.
The cost function $C$ of OPF is usually assumed to be convex in its variables.  
The difficulty of OPF formulated here
 thus arises  from the nonconvex feasible sets $\mathbb V$ for BIM
and $\mathbb{\tilde X}$ for BFM.
The basic approach to deriving convex relaxations of OPF is to design convex
supersets of (equivalent sets of)
$\mathbb V$ or $\mathbb{\tilde X}$ and minimize the same cost
function over these  supersets.  Different choices of convex supersets
lead to different relaxations, but they all provide a lower bound to OPF.
If every optimal solution of a convex relaxation happens to lie in $\mathbb V$
or $\mathbb{\tilde X}$ then it is also feasible and hence optimal for the original
OPF.   In this case we say the realxation is exact.

In this section we present three characterizations of the feasible set $\mathbb V$
in BIM.    These characterizations naturally
suggest convex supersets and semidefinite relaxations of OPF, and we prove equivalence relations
among them.   In the next section we treat BFM.
In Part II of the paper we discuss sufficient conditions that guaranteed exact relaxations.

\subsection{Preliminaries}

Since OPF is a nonconvex QCQP  
there is a standard semidefinite relaxation through the equivalence
relation: for any Hermitian matrix $M$, 
$V^H M V = $ tr $M VV^H = $ tr $MW$ for a psd rank-1 matrix $W$.
Applying this transformation to the QCQP formulation \eqref{eq:OPFbim.2} leads to an
equivalent problem of the form:
\bqn
\min_{W\in \mathbb S^{n+1}}  & &  \text{ tr } C W
\\
\text{subject to} & & \text{ tr } C_l W \ \leq \ b_l,
			\ W\succeq 0, \ \text{ rank }W = 1
\eqn
for appropriate Hermitian matrices $C_l$ and real numbers $b_l$.
This problem is equivalent to \eqref{eq:OPFbim.2} because given a psd rank-1 solution $W$,
a unique solution $V$ of \eqref{eq:OPFbim.2} can be recovered through 
rank-1 factorization $W = VV^H$.
Unlike \eqref{eq:OPFbim.2} which is quadratic in $V$ this problem is convex in 
$W$ except the nonconvex rank-1 constraint.  Removing the rank-1 constraint yields
the standard SDP relaxation.  

We now generalize this intuition to 
characterize the feasible set $\mathbb V$ in \eqref{eq:defV} in terms of partial matrices.  
These characterizations lead naturally to  SDP, chordal, and second-order cone
program (SOCP) relaxations of OPF
in BIM, as shown in \cite{Bose-2012-BFMe-Allerton, Bose-2014-BFMe-TAC}.

We start with some basic definitions on partial matrices and their completions;
see e.g. \cite{Grone1984, Fukuda99exploitingsparsity, Nakata2003} for more details.
Fix any connected undirected graph $F$ with $n$ vertices and $m$ edges connecting distinct 
vertices.\footnote{In this subsection we abuse notation and use $n, m$ to denote general integers
unrelated to the number of buses or lines in a power network.}
A \emph{partial matrix} $W_F$ is a set of $2m+n$ complex numbers \emph{defined on $F$}:
\bqn
W_F  & := &   \left\{ \ [W_F]_{jj}, [W_F]_{jk}, [W_F]_{kj}   
	\ | \ \text{nodes $j$ and edges $(j,k)$ of $F$} \ \right\}
\eqn
$W_F$ can be interpreted as a matrix with entries partially specified by these complex
numbers.
If $F$ is a complete graph (in which there is an edge between every pair of vertices)
 then $W_F$ is a fully specified $n\times n$ matrix.
A \emph{completion} $W$ of $W_F$ is  any fully specified $n\times n$ matrix that agrees with 
$W_F$ on  graph $F$, i.e., 
\bqn
[W]_{jj} \ = \ [W_F]_{jj}, \ [W]_{jk} \ = \ [W_F]_{jk}\ \ \text{ for } j, (j,k) \in F
\eqn
Given an $n\times n$ matrix $W$ we use $W_F$ to denote the \emph{submatrix of $W$ on $F$,} i.e., 
the partial matrix consisting of the entries of $W$ defined on graph $F$.
If $q$ is a clique (a fully connected subgraph) of $F$ then let $W_F(q)$ denote the
fully-specified principal submatrix of $W_F$ defined on $q$.
We extend the definitions of Hermitian, psd, and rank-1 for matrices to partial matrices, as follows.
A partial matrix $W_F$ is \emph{Hermitian}, denoted by $W_F = W_F^H$,
 if $[W_F]_{jk} = [W_F]_{kj}^H$ for all $(j,k)\in F$;
it is \emph{psd}, denoted by $W_F \succeq 0$, if $W_F$ is Hermitian and 
the principal submatrices $W_F(q)$ are psd for
all cliques $q$ of $F$; it is \emph{rank-1}, denoted by rank $W_F = 1$, if the principal submatrices 
$W_F(q)$ are rank-1 for all cliques $q$ of $F$.
We say $W_F$ is \emph{$2\times 2$ psd (rank-1)} if, for all edges $(j,k)\in F$, 
the $2\times 2$ 
principal submatrices
\bqn
W_F(j,k) & := & \begin{bmatrix}
		[W_F]_{jj}  & [W_F]_{jk}  \\   [W_F]_{kj}  &  [W_F]_{kk}
		\end{bmatrix}
\eqn
are psd (rank-1), denoted  by $W_F(j,k) \succeq 0$ (rank $W_F(j,k) = 1)$.
$F$ is a \emph{chordal graph} if either $F$ has no cycle or all its minimal cycles (ones without
chords) are of length three.
A \emph{chordal extension} $c(F)$ of  $F$ is a chordal graph that contains $F$, i.e., $c(F)$ 
has the same vertex set as $F$ but an edge set that is a superset of $F$'s edge set.
In that case we call the partial matrix $W_{c(F)}$ a \emph{chordal extension} of the
partial matrix $W_F$.
Every graph $F$ has a chordal extension, generally nonunique.  In particular a complete
supergraph of $F$ is a trivial chordal extension of $F$.

For our purposes chordal graphs are important because of the result 
\cite[Theorem 7]{Grone1984} that every psd partial matrix has a psd completion
if and only if the underlying graph is chordal.
When a positive definite completion  exists,
there is a \emph{unique} positive definite completion, in the class of all
positive definite completions, whose determinant is maximal.   
Theorem \ref{thm:rank1} below extends this to rank-1 partial matrices.

\subsection{Feasible sets}
We can now characterize the feasible set $\mathbb V$ of OPF defined in \eqref{eq:defV}. 
Recall the undirected connected graph $G = (N^+, E)$ that models a power network.
Given a voltage vector $V\in \mathbb V$ define a partial matrix $W_G := W_G(V)$: 
for $j\in N^+$ and $(j,k)\in E$,
\begin{subequations}
\bq
[W_G]_{jj} & := & |V_j|^2
\\
\left[W_G \right]_{jk} & := & V_j V_k^H   \ =: \ [W_G]_{kj}^H
\eq
\label{eq:defpm}
\end{subequations}
Then the constraints
\eqref{eq:bimopf.1} and \eqref{eq:opfv}  imply that the partial matrix $W_G$ satisfies
 \footnote{The constraint
\eqref{eq:opfW.1} can also be written compactly in terms of the admittance matrix 
$Y$ as in \cite{Zhang2013}:
\bqn
\underline{s} \ \leq \  \text{diag } \left( WY^H \right) \ \leq \ \overline{s}
\eqn
}
\begin{subequations}
\bq
\underline{s}_j \  \leq\!\!  \sum_{k: (j,k)\in E}\!\! y_{jk}^H\, \left( [W_G]_{jj} - [W_G]_{jk} \right) 
 \leq  \ \overline{s}_j,  \ \ j \in N^+
\label{eq:opfW.1}
\\
\underline{v}_j \ \leq \ [W_G]_{jj} \ \leq \ \overline{v}_j,  \qquad\qquad j \in N^+
\label{eq:opfW.2}
\eq
\label{eq:opfW}
\end{subequations}
Following Section \ref{subsec:opfqcqp} these  constraints can also be written
in a (partial) matrix form as:
\bqn
\underline{p}_j & \leq\  \text{tr}\ \Phi_j W_G \ \leq & \overline{p}_j
\\
\underline{q}_j & \leq\  \text{tr}\ \Psi_j W_G \ \leq & \overline{q}_j
\\
\underline{v}_j & \leq\  \text{tr}\ J_j W_G \ \leq & \overline{v}_j
\eqn

The converse is not always true: given a partial matrix $W_G$ that satisfies \eqref{eq:opfW}
it is not always possible to recover a voltage vector $V$ in $\mathbb V$.
Indeed this is possible if and only if $W_G$ has a completion $W$ that is psd rank-1,
because in that case $W$ satisfies \eqref{eq:opfW} since $y_{jk}=0$ if $(j,k)\not\in E$ and 
it can be uniquely factored as $W= VV^H$ with $V\in \mathbb V$.  
We hence seek conditions additional to  \eqref{eq:opfW} on the partial matrix $W_G$ that guarantee  that
it has a psd rank-1 completion $W$ from which $V\in \mathbb V$ can be recovered.
Our first key result provides such a characterization.

We say that a partial matrix $W_G$ satisfies the {\em cycle condition} if for every 
cycle $c$ in $G$
\bq
\sum_{(j,k)\in c}\ \angle [W_G]_{jk} & = & 0   \ \ \mod 2\pi
\label{eq:cyclecond.2}
\eq
When $\angle [W_G]_{jk}$ represent voltage phase differences across each line then the
cycle condition imposes that they sum to zero (mod $2\pi$) around any cycle.
The next theorem, proved in \cite[Theorem 3]{Bose-2012-BFMe-Allerton} and \cite{Bose-2014-BFMe-TAC},
implies that $W_G$ has a psd rank-1 completion $W$ if and only if
$W_G$ is $2\times 2$ psd rank-1 on $G$ and satisfies the cycle condition \eqref{eq:cyclecond.2}, 
if and only if it has a chordal extension $W_{c(G)}$ that is psd rank-1. \footnote{The theorem also holds 
with psd replaced by negative semidefinite.}

%
%
Consider the following conditions on $(n+1)\times (n+1)$ matrices $W$
and partial matrices $W_{c(G)}$ and $W_G$:
\bq
W\succeq 0,  \!\! & \!\!\!\!\!\!\! & \!\!\!\!  \text{rank }W = 1
\label{eq:Wrank1}
\\
W_{c(G)}\succeq 0,  \!\!  & \!\!\!\!\!\!\!& \!\!\!\!  \text{rank }W_{c(G)} = 1
\label{eq:WcGrank1}
\\
W_G(j,k) \succeq 0,  \!\!  & & \!\!\!\!  \text{rank }W_G(j,k) = 1, \ \ \ \ (j,k)\in E,
\label{eq:2x2rank1}
\eq
\begin{theorem}
\label{thm:rank1}
Fix a graph $G$ on $n+1$ nodes and any chordal extension $c(G)$ of $G$.
Assuming $W_{jj}>0$, $\left[ W_{c(G)} \right]_{jj}>0$ and 
$\left[ W_G \right]_{jj}>0$, $j\in N^+$, we have:
\bee
\item[(1)] Given an $(n+1)\times (n+1)$ matrix $W$ that satisfies \eqref{eq:Wrank1},
	its submatrix $W_{c(G)}$ satisfies  \eqref{eq:WcGrank1}.
\item[(2)] Given a partial matrix $W_{c(G)}$ that satisfies  \eqref{eq:WcGrank1},
	its submatrix $W_G$ satisfies \eqref{eq:2x2rank1} and the cycle
	condition \eqref{eq:cyclecond.2}.
\item[(3)] Given a partial matrix $W_G$ that satisfies 
	\eqref{eq:2x2rank1} and the cycle condition \eqref{eq:cyclecond.2}, 
	there is a completion $W$ of $W_G$ that satisfies \eqref{eq:Wrank1}.
\eee
%
 \end{theorem}

Informally Theorem \ref{thm:rank1} says that \eqref{eq:Wrank1} is equivalent
to \eqref{eq:WcGrank1} is equivalent to \eqref{eq:2x2rank1}$+$\eqref{eq:cyclecond.2}.
It characterizes a property of the full matrix $W$
(rank $W = 1$) in terms of its submatrices $W_{c(G)}$ and $W_G$.
This is important because the submatrices are typically much smaller than $W$ 
for large sparse networks and much easier to compute.
The theorem thus allows us to solve simpler problems in terms of partial matrices
as we now explain.

Define the set of Hermitian matrices:
\begin{equation}
\begin{array}{rl}
 \mathbb W \ := \  \{W \in \mathbb S^{n+1} \ \vert \ & \!\!\!\!\! W \text{ satisfies } 
		\eqref{eq:opfW}, \eqref{eq:Wrank1}  \}
\end{array}
\label{eq:defW}
\end{equation}
Fix any chordal extension $c(G)$ of $G$ and define the set of Hermitian {partial}
matrices $W_{c(G)}$:
\begin{equation}
\begin{array}{rl}
 \mathbb W_{c(G)} \ := \ \{W_{c(G)}  \ \vert \ & \!\!\!\!\! W_{c(G)} \text{ satisfies } \eqref{eq:opfW},  \eqref{eq:WcGrank1} 
  \}
\end{array}
\label{eq:defWcG}
\end{equation}
Finally define the set of Hermitian partial matrices $W_G$:
\begin{equation}
\begin{array}{rl}
 \mathbb W_G \ := \ \{W_{G}  \ \vert & \!\!\!\!\! W_G \text{ satisfies } \eqref{eq:opfW}, 
 			\eqref{eq:cyclecond.2},  \eqref{eq:2x2rank1}     \}
\end{array}
\label{eq:defWG}
\end{equation}
Note that the definition of psd for partial matrices implies that $W_{c(G)}$ and $W_G$ are Hermitian.
The assumption $\underline v_j>0, j\in N^+$ implies that all matrices or partial matrices have
strictly positive diagonal entries.

Theorem \ref{thm:rank1} implies that given a partial matrix $W_{c(G)} \in \mathbb W_{c(G)}$
or a partial matrix $W_G\in \mathbb W_G$ there is a psd rank-1 completion $W\in \mathbb W$ from
which a solution $V\in \mathbb V$ of OPF can be recovered.
In fact we know more: given any Hermitian partial matrix $W_G$ (not necessarily in $\mathbb W_G$), 
the set of \emph{all} 
completions of $W_G$ that satisfies the condition in Theorem \ref{thm:rank1}(3) 
consists of  a {single} psd rank-1 matrix
and infinitely many indefinite non-rank-1 matrices;
 see \cite[Theorems 5 and 8]{Bose-2012-BFMe-Allerton} and discussions therein. 
Hence the psd rank-1 completion $W$ of a $W_G \in \mathbb W_G$ is unique.
\begin{corollary}
\label{coro:eq}
Given a partial matrix $W_{c(G)} \in \mathbb W_{c(G)}$ or $W_G \in \mathbb W_G$
there is a unique psd rank-1 completion $W \in \mathbb W$.
\end{corollary}

Recall that two sets $A$ and $B$ are {equivalent} ($A\equiv B$)
if there is a bijection between them.  
Even though $\mathbb W, \mathbb W_{c(G)}, \mathbb W_G$ are different kinds of spaces
Theorem \ref{thm:rank1} and Corollary \ref{coro:eq}  imply that they are all equivalent 
to the feasible set of OPF.
\begin{theorem}
\label{thm:FS}
$\mathbb V \equiv \mathbb W \equiv \mathbb W_{c(G)} \equiv \mathbb W_G$.
\end{theorem}
\vspace{0.05in}

Theorem \ref{thm:FS} suggests three equivalent problems to OPF.
We assume the cost function $C(V)$ in OPF depends on $V$ only through the
partial matrix $W_G$ defined in \eqref{eq:defpm}.   For example if the cost is total real line loss in the 
network then $C(V) = \sum_j \text{Re } s_j = \sum_j \sum_{k:(j,k)\in E} 
\text{Re} \left([W_G]_{jj} - [W_G]_{jk} \right) y_{jk}^H$.  If the
cost is a weighted sum of real generation power then
$C(V) = \sum_j \left(c_j \, \text{Re } s_j + p_j^{d} \right)$ where $p_j^{d}$ are the given real power
demands at buses $j$; again $C(V)$ is a function of the partial matrix $W_G$.
Then Theorem \ref{thm:FS} implies that OPF \eqref{eq:OPFbim} is equivalent to 
\bq
\underset{W}{\text{min}}   \ C({W}_G) \ & \text{subject to} &   W \in \hat{\mathbb W}
\label{eq:bimeqp}
\eq
where $\hat{\mathbb W}$ is any one of the sets $\mathbb W, \mathbb W_{c(G)}, \mathbb W_G$.
Specifically, given an optimal solution $W^\text{opt}$ in $\mathbb W$, it can be uniquely
decomposed into $W^\text{opt} = V^\text{opt}(V^\text{opt})^H$.  Then $V^\text{opt}$ is in $\mathbb V$ 
and an optimal solution of OPF \eqref{eq:OPFbim}.   Alternatively given an optimal solution $W_F^\text{opt}$ in
$\mathbb W_{c(G)}$ or $\mathbb W_G$, Corollary \ref{coro:eq} guarantees that 
$W_F^\text{opt}$ has a unique psd rank-1 completion $W^\text{opt}$ in $\mathbb W$ from which an optimal 
$V^\text{opt}\in \mathbb V$ can be recovered.
In fact given a partial matrix $W_G \in \mathbb W_G$ (or $W_{c(G)} \in \mathbb W_{c(G)}$)
there is a more direct construction of a feasible solution $V\in \mathbb V$ of OPF than through
its completion; see Section \ref{subsec:sr}.

\subsection{Semidefinite relaxations}
Hence solving OPF \eqref{eq:OPFbim} is equivalent to solving \eqref{eq:bimeqp}
over any of $\mathbb W, \mathbb W_{c(G)}, \mathbb W_G$ for an appropriate matrix variable.
The difficulty with solving \eqref{eq:bimeqp} is that the feasible sets 
$\mathbb W$, $\mathbb W_{c(G)}$, and $\mathbb W_G$ are still
nonconvex due to the rank-1 constraints and the cycle condition \eqref{eq:cyclecond.2}.   
Their removal leads to  SDP, chordal, and SOCP relaxations of OPF
respectively.

Relax $\mathbb W$, $\mathbb W_{c(G)}$ and $\mathbb W_G$  to the following convex supersets:
\bqn
 \mathbb W^+ & \!\!\! := \!\!\! &  \{W \in \mathbb S^{n+1}  \, \vert \, W_G \text{ satisfies } 
		\eqref{eq:opfW},   W\succeq 0 \}
\\
 \mathbb W_{c(G)}^+ & \!\!\! := \!\!\! &  \{W_{c(G)} \, \vert \, W_G \text{ satisfies } 
		\eqref{eq:opfW},   W_{c(G)} \succeq 0 \}
\\
 \mathbb W_G^+ & \!\!\! := \!\!\! &  \{W_G  \, \vert \, W_G \text{ satisfies } 
		\eqref{eq:opfW},   W_G(j,k) \succeq 0, \ (j,k)\in E \}
\eqn
Define the problems:
\\
\noindent
\textbf{OPF-sdp:}
\bq
\label{eq:bimOPF-sdp}
\hspace{-0.315in}
\underset{W}{\text{min}}   \ C(W_G) \ & \text{subject to} &   W \in \mathbb W^+
\eq
\noindent
\textbf{OPF-ch:}
\bq
\label{eq:bimOPF-ch}
\underset{W_{c(G)}}{\text{min}}   \ C(W_G) \ & \text{subject to} &   W_{c(G)} \in \mathbb W_{c(G)}^+
\eq
\noindent
\textbf{OPF-socp:}
\bq
\label{eq:bimOPF-socp}
\hspace{-0.285in}
\underset{W_G}{\text{min}}   \ C(W_G) \ & \text{subject to} &   W_G \in \mathbb W_G^+
\eq
The condition $W_G(j,k)\succeq 0$ in the definition of $\mathbb W_G^+$ is 
equivalent to $\left[ W_G \right]_{jk} = \left[ W_G \right]_{kj}^H$ and
(recall the assumption $\underline v_j>0, j\in N^+$)
\bqn
\quad [W_G]{jj} > 0, \ [W_G]_{kk} > 0, \  [W_G]_{jj}[W_G]_{kk} \geq \left| [W_G]_{jk} \right|^2
\eqn
This is a second-order cone and
hence OPF-socp is indeed an SOCP in the rotated form.  

\begin{remark}
\label{remark:bioCoReBIM}
\emph{Literature.}
SOCP relaxation for OPF seems to be first proposed in \cite{Jabr2006} 
for the bus injection model \eqref{eq:bim.1}, 
and in \cite{Farivar2011-VAR-SGC, Farivar-2013-BFM-TPS} for the branch flow model \eqref{eq:bfm}
as explained in the next section.
By defining a new set of variables $v_j := |V_j|^2$, $R_{jk} := |V_j| |V_k| \cos (\theta_j- \theta_k)$,
and $I_{jk} := |V_j| |V_k| \sin (\theta_j- \theta_k)$ where $\theta_j := \angle V_j$, \cite{Jabr2006} 
rewrites the bus injection model \eqref{eq:bim.1} in the complex domain as a set of linear equations 
in these new variables in the real domain and the following quadratic equations:
\bqn
v_j v_k & = & R_{jk}^2 + I_{jk}^2
\eqn
Relaxing these equalities to $v_j v_k \geq R_{jk}^2 + I_{jk}^2$ enlarges the solution set to a
second-order cone that is equivalent to $\mathbb W_G^+$ in this paper.
SDP relaxation is first proposed in \cite{Bai2008} for the bus injection model and analyzed
in \cite{Lavaei2012}.
Chordal relaxation for OPF is first proposed in \cite{Bai2011, Jabr2012} and analyzed in
\cite{MolzahnLesieutre2013, Bose-2014-BFMe-TAC}.
\end{remark}

\subsection{Solution recovery}
\label{subsec:sr}

When the convex relaxations OPF-sdp, OPF-ch, OPF-socp are exact, i.e., if their optimal
solutions $W^\text{sdp}$, $W_{ch}^\text{ch}$, $W_G^\text{socp}$ happen to lie in
$\mathbb W$, $\mathbb W_{c(G)}$, $\mathbb W_G$ respectively, then an optimal solution
$V^\text{opt}$ of the original OPF can be recovered from these solutions.
Indeed the recovery method works not just for an optimal solution, but any feasible solution
that lies in $\mathbb W$, $\mathbb W_{c(G)}$ or $\mathbb W_G$.
Moreover, given a $W\in \mathbb W$ or a $W_{c(G)} \in \mathbb W_{c(G)}$,
the construction of $V$ depends on $W$ or $W_{c(G)}$ only through their submatrix 
$W_G$.   We hence describe the method for recovering the unique $V$ from a 
$W_G$, which may be a partial matrix in  $\mathbb W_G$ or the submatrix of a (partial) matrix
in $\mathbb W$ or $\mathbb W_{c(G)}$.

Let $T$ be an arbitrary spanning tree of ${G}$ rooted at bus 0.   
Let $\mathbb P_{j}$ denote the unique path from node $0$ to node $j$ in $T$.
Recall that $V_0 = 1 \angle 0^\circ$ without loss of generality.   
For $j = 1, \dots, n$, let 
\bqn
|V_j|  & := &  \sqrt{\left[ W_G \right]_{jj}}  
\\
\angle V_j & := &  - \sum_{(i,k)\in \mathbb P_{j}} \angle \left[ W_G \right]_{ik}
\eqn
Then it can be checked that $V$ is in \eqref{eq:defV} and feasible for OPF.

\subsection{Tightness of relaxations}
Since $\mathbb W \subseteq \mathbb W^+$, 
$\mathbb W_{c(G)} \subseteq \mathbb W_{c(G)}^+$, 
$\mathbb W_G \subseteq \mathbb W_G^+$, the relaxations 
OPF-sdp, OPF-ch, OPF-socp all provide lower bounds on OPF 
\eqref{eq:OPFbim} in light of Theorem \ref{thm:FS}.
%
OPF-socp is the simplest computationally.  OPF-ch usually requires more computation
than OPF-socp but much less than OPF-sdp for large sparse networks (even though 
OPF-ch can be as complex as OPF-sdp in the worse case \cite{Fukuda99exploitingsparsity, Nakata2003}).
The relative tightness of the relaxations depends on the network topology.
For a general mesh network OPF-sdp is as tight a relaxation as OPF-ch and they are strictly
tighter than OPF-socp.   For a tree (radial) network the hierarchy collapses and all three are equally
tight.    We now make this precise.

Consider their feasible sets $\mathbb W^+$, $\mathbb W_{c(G)}^+$ and $\mathbb W_G^+$.   
We say that a set $A$ is an \emph{effective subset} of a set $B$, denoted by
$A \sqsubseteq B$, if, given a (partial) matrix $a\in A$, there is a (partial)
matrix $b\in B$ that has the same cost $C(a) = C(b)$.   We say $A$ is \emph{similar to} $B$, denoted by $A\simeq B$, if
$A\sqsubseteq B$ and $B \sqsubseteq A$.   Note that $A\equiv B$ implies
$A \simeq B$ but the converse may not be true.
The feasible set of OPF \eqref{eq:OPFbim} is an effective subset of the feasible
sets of the relaxations; moreover these relaxations have similar feasible sets when
the network is radial.    This is a slightly different formulation
of the same results in \cite{Bose-2012-BFMe-Allerton, Bose-2014-BFMe-TAC}.
\begin{theorem}
\label{thm: bimR}
$\mathbb V \sqsubseteq \mathbb W^+ \simeq \mathbb W_{c(G)}^+ \sqsubseteq \mathbb W_G^+$.
If $G$ is a tree then 
	$\mathbb V \sqsubseteq \mathbb W^+ \simeq \mathbb W_{c(G)}^+ \simeq \mathbb W_G^+$.
\end{theorem}

Let $C^\text{opt}, C^\text{sdp}, C^\text{ch}, C^\text{socp}$ be the optimal values of OPF \eqref{eq:OPFbim},
OPF-sdp \eqref{eq:bimOPF-sdp}, OPF-ch \eqref{eq:bimOPF-ch}, OPF-socp \eqref{eq:bimOPF-socp}
 respectively.   Theorem \ref{thm:FS} and 
Theorem \ref{thm: bimR}  directly imply
\begin{corollary}
\label{coro:bimR}
$C^\text{opt} \geq C^\text{sdp} = C^\text{ch} \geq C^\text{socp}$.
If $G$ is a tree then $C^\text{opt} \geq C^\text{sdp} = C^\text{ch} = C^\text{socp}$.
\end{corollary}

\begin{remark}
\label{remark:socp1}
\emph{Tightness.}
Theorem \ref{thm: bimR} and Corollary \ref{coro:bimR}
 imply that for radial networks one should always solve OPF-socp since it is the tightest
and the simplest relaxation of the three.
For mesh networks there is a tradeoff between OPF-socp and OPF-ch/OPF-sdp: the latter is 
tighter but requires heavier computation. 
Between OPF-ch and OPF-sdp, OPF-ch is usually preferable as they are equally tight
but OPF-ch is usually much faster to solve for large sparse networks.
See \cite{Bai2011, Jabr2012, Bose-2014-BFMe-TAC, MolzahnLesieutre2013, Andersen2013} for 
numerical studies that compare these relaxations. 
\end{remark}

\subsection{Chordal relaxation}

 Theorem \ref{thm:rank1} through Corollary \ref{coro:bimR} apply to \emph{any}
chordal extension $c(G)$ of $G$.   The choice of $c(G)$ does not affect the optimal
value of the chordal relaxation but determines its complexity.
Unfortunately the optimal choice that minimizes the complexity
of OPF-ch is NP-hard to compute.

This difficulty is due to two conflicting factors in choosing a $c(G)$.
Recall that the constraint $W_{c(G)}\succeq 0$ in the definition of $\mathbb W_{c(G)}^+$
consists of multiple constraints that the principal submatrices
$W_{c(G)}(q) \succeq 0$, one  for each (maximal) clique $q$ of $c(G)$.
When two cliques $q$ and $q'$ share a node their submatrices $W_{c(G)}(q)$ and $W_{c(G)}(q')$
share entries that must be decoupled by introducing auxiliary variables and equality constraints
on these variables.   The choice of $c(G)$ determines the number and sizes of these submatrices 
$W_{c(G)}(q)$ as well as the numbers of auxiliary variables and additional decoupling constraints.
On the one hand if $c(G)$ contains few cliques $q$ then the submatrices
$W_{c(G)}(q)$ tend to be large and expensive to compute (e.g. if $c(G)$ is the complete graph then 
there is a single clique, but $W_{c(G)} = W$ and OPF-ch is identical to OPF-sdp).
On the other hand if $c(G)$ contains many small cliques $q$ then there tends to be
more overlap and chordal relaxation tends to require more decoupling constraints.
Hence choosing a good chordal extension $c(G)$ of $G$ is important but nontrivial. 
See  \cite{Fukuda99exploitingsparsity, Nakata2003} and references therein for  
methods to compute efficient chordal relaxations of general QCQP.
For OPF \cite{MolzahnLesieutre2013} proposes effective techniques to 
reduce the number of cliques in its chordal relaxation.
To further reduce the problem size \cite{Andersen2013} proposes to carefully drop some 
of the decoupling constraints,  though the resulting relaxation can  be weaker.

%% file: 1-relaxationsBFM.tex
\section{Feasible sets and relaxations: BFM}
\label{sec:rBFM}

We now present an SOCP relaxation of OPF in BFM
proposed in \cite{Farivar2011-VAR-SGC, Farivar-2013-BFM-TPS} in two steps.
We first relax the phase angles of $V$ and  $I$ in \eqref{eq:bfm}
and then we relax a set of quadratic equalities to inequalities.   
This derivation pinpoints the difference between
radial and mesh topologies.  It motivates a recursive version of BFM for
radial networks (Section \ref{sec:BFMt}) and the use of phase shifters for 
convexification of mesh networks (Part II \cite{Low2014b}).

\subsection{Feasible sets}
\label{subsec:rBFMfs}

Consider the following set of equations in the variables $x := (S, \ell, v, s)$ in 
$\mathbb R^{3(m+n+1)}$:\footnote{The use of complex variables is only a shorthand and
should be interpreted as operations in real variables.  For instance \eqref{eq:mdf.1} is a
shorthand for
\bqn
\sum_{k: j\rightarrow k} P_{jk} & = & \sum_{i: i\rightarrow j} \left( P_{ij} - r_{ij} \ell_{ij} \right) + p_j
\\
\sum_{k: j\rightarrow k} Q_{jk} & = & \sum_{i: i\rightarrow j} \left( Q_{ij} - x_{ij} \ell_{ij} \right) + q_j
\eqn
}
\begin{subequations}
\bq
\sum_{k: j\rightarrow k} S_{jk} & = & \sum_{i: i\rightarrow j} \left( S_{ij} - z_{ij} \ell_{ij} \right)
		 + s_j, \quad j \in N^+
\label{eq:mdf.1}
\\
v_j - v_k & = & 2\, \text{Re} \left(z_{jk}^H S_{jk} \right) - |z_{jk}|^2 \ell_{jk}, \ \ j\rightarrow k \in \tilde E
\label{eq:mdf.2}
\\
v_j \ell_{jk} & = & |S_{jk}|^2, \quad j\rightarrow k \in \tilde E
\label{eq:mdf.3}
\eq
\label{eq:mdf}
\end{subequations}
and define the solution set as:
\bqn
{\mathbb X}_{nc} & := & \{ x \in \mathbb R^{3(m+n+1)} \ \vert \ x \text{ satisfies } 
		\eqref{eq:opfv}, \eqref{eq:opfs}, \eqref{eq:mdf} \}
\eqn
Note that the vector $v$ includes $v_0$ and $s$ includes $s_0$. 
The model \eqref{eq:mdf} is first proposed in \cite{Baran1989a, Baran1989b}.
\footnote{The original model, called the DistFlow equations, in \cite{Baran1989a, Baran1989b} 
is for radial (distribution) networks, but its extension here to mesh networks  is trivial.}
It can be derived as a relaxation of BFM \eqref{eq:bfm}  as follows.
Taking the squared magnitude  of \eqref{eq:bfm.2} and replacing $|V_j|^2$ and $|I_{jk}|^2$ by
$v_j$ and $\ell_{jk}$ respectively yield \eqref{eq:mdf.3}. 
To obtain \eqref{eq:mdf.2}, use \eqref{eq:bfm.1}--\eqref{eq:bfm.2} to write $V_k = V_j - z_{jk} S_{jk} V_j^{-1}$ and 
take the squared magnitude on both sides to eliminate the phase angles of $V$ and $I$. 
These operations define a mapping $h: \mathbb C^{2(m + n +1)} \rightarrow \mathbb R^{3(m+n+1)}$ by:
for any $\tilde x = (S, I, V, s)$, $h(\tilde x) := (S, \ell, v, s)$ with 
$\ell_{jk} = |I_{jk}|^2$ and $v_j = |V_j|^2$.

Throughout this paper we assume the cost function $C(\tilde x)$ in OPF 
\eqref{eq:OPFbfm} depends on $\tilde x$ only through 
$x := h(\tilde x)$.
For example for total real line loss 
$C(\tilde x) = \sum_{(j, k) \in \tilde{E}} \, \text{Re}\, z_{jk} \ell_{jk}$.
If the cost is a weighted sum of real generation power then
$C(\tilde x) = \sum_j (c_j p_j + p_j^{d})$ where $p_j$ are the real parts of
$s_j$ and $p_j^{d}$ are the given real power
demands at buses $j$; again $C(\tilde{x})$ depends only on $x$.

Then the model \eqref{eq:mdf} is a relaxation of BFM \eqref{eq:bfm}
in the sense that the feasible set $\mathbb{\tilde X}$ of OPF in \eqref{eq:OPFbfm} is an effective
subset of $\mathbb X_{nc}$,  
 $\mathbb{\tilde X} \sqsubseteq \mathbb X_{nc}$, since $h(\mathbb{\tilde X}) \subseteq \mathbb X_{nc}$.
We now characterize the subset of $\mathbb X_{nc}$ that is equivalent to 
$\mathbb{\tilde X}$.   

Given an $ x := (S, \ell, v, s) \in \mathbb R^{3(m+n+1)}$ define $\beta({x}) \in \mathbb R^m$ by
\bq
\beta_{jk}( x) & := & \angle \left( v_j - z_{jk}^H S_{jk} \right), \quad  j \rightarrow k\in \tilde E
\label{eq:defb}
\eq
Even though $x$ does not include phase angles of $V$, $x$
\emph{implies} a phase difference across each line $j\rightarrow k \in \tilde E$
given by $\beta_{jk}(x)$.  
The subset of $\mathbb X_{nc}$ that is equivalent to $\mathbb{\tilde X}$ are those $x$ for which there exists
$\theta$ such that $\theta_j - \theta_k = \beta_{jk}(x)$.
To state this precisely let $B$ be the $m\times n$ (transposed) reduced incidence 
matrix of $\tilde{G}$:
\bqn
B_{lj} & = & \begin{cases}
		1 & \text{ if edge $l \in \tilde E$ leaves node $j$} \\
		-1 & \text{ if edge $l \in \tilde E$ enters node $j$} \\
		0 & \text{ otherwise}
		\end{cases}
		\quad\quad  
\eqn
where $j\in N$.
Consider the set of $x \in \mathbb X_{nc}$ such that 
\bq
\!\!\!\!\!\!
\exists \theta \text{ that solves }\ 
	B\theta &\!\!\! = \!\!\! & \beta(x)  \mod 2\pi
\label{eq:cyclecond.1}
\eq
i.e., $\beta(x)$ is in the range space of $B$ (mod $2\pi$).
A solution $\theta(x)$, if exists, is unique in  $(-\pi, \pi]^n$.
Define the set
 \bqn
{\mathbb X} & \!\!\!\! := \!\!\!\! &  \{ x \in \mathbb R^{3(m+n+1)} \ \vert \ x \text{ satisfies } 
		\eqref{eq:opfv}, \eqref{eq:opfs}, \eqref{eq:mdf},  \eqref{eq:cyclecond.1}  \}
\eqn
The following result characterizes the feasible set $\mathbb{\tilde X}$ of OPF in BFM
and follows from \cite[Theorems 2, 4]{Farivar-2013-BFM-TPS}.
\begin{theorem}
\label{thm:eqX}
$\mathbb{\tilde X} \equiv \mathbb X \subseteq \mathbb X_{nc}$.
\end{theorem}
The bijection between $\mathbb{\tilde X}$ and $\mathbb X$ is given by 
$h$ defined above restricted to $\mathbb{\tilde X}$.
Its inverse $h^{-1}(S, \ell, v, s) = (S, I, V, s)$ is defined on $\mathbb X$ 
in terms of $\theta(x)$ by:
\begin{subequations}
\bq
V_j & := & \sqrt{v_j} \ e^{\ii \theta_j(x)},  \quad j \in N
\\
I_{jk} & := & \sqrt{\ell_{jk}}\ e^{\ii (\theta_j(x) - \angle S_{jk})}, \  j\rightarrow k \in \tilde E
\eq
\label{eq:defhinv}
\end{subequations}

The condition \eqref{eq:cyclecond.1} is equivalent to the cycle condition \eqref{eq:cyclecond.2}
in the bus injection model.   To see this
 fix any spanning tree $T = (N, E_T)$ of the (directed) graph $\tilde G$.   
We can assume without loss of generality 
(possibly after re-labeling the links) that $E_T$ consists of links $l = 1, \dots, n$.  
Then $B$ can be partitioned into
\bqn
B & = & \begin{bmatrix}
		B_T  \\ B_{\perp}
		\end{bmatrix}
\eqn
where the $n\times n$ submatrix $B_T$ corresponds to links in $T$ and the 
$(m-n) \times n$ submatrix $B_{\perp}$ corresponds to links in $T^\perp := G\setminus T$.
Similarly partition $\beta(x)$ into
\bqn
\beta(x) & = & \begin{bmatrix}
		\beta_T (x) \\  \beta_{\perp} (x)
		\end{bmatrix}
\eqn
The next result, proved in \cite[Theorems 2 and 4]{Farivar-2013-BFM-TPS}, provides 
a more explicit characterization of \eqref{eq:cyclecond.1} in terms of $\beta(x)$.  
When it holds this characterization has the same interpretation
of the cycle condition in \eqref{eq:cyclecond.2}: 
the voltage angle differences implied by $x$ sum to zero (mod $2\pi$) around any cycle.
 Formally let $\tilde \beta$ be the extension of $\beta$ from directed to undirected
 links: for each $j\rightarrow k \in \tilde E$ let $\tilde \beta_{jk}(x) := \beta_{jk}(x)$ and
 $\tilde \beta_{kj}(x) := - \beta_{jk}(x)$.   We say $c := (j_1, \dots, j_K)$ is an \emph{undirected}
 cycle if, for each $k=1, \dots, K$, either $j_k \rightarrow j_{k+1} \in \tilde E$ or 
 $j_{k+1} \rightarrow j_k \in \tilde E$ with the interpretation that $j_{K+1} := j_1$;  $(j_k, j_{k+1}) \in c$  
 denotes one of these links.
\begin{theorem}
\label{thm:anglecond}
An $x \in \mathbb X_{nc}$ satisfies \eqref{eq:cyclecond.1} if and only if
around each undirected cycle $c$ we have
\bq
\sum_{(j,k) \in c} \, \tilde \beta_{jk}(x) & = & 0 \qquad \text{ mod } 2\pi
\label{eq:cyclecond.3}
\eq
In that case $\theta(x) = \mathcal P \left(B_T^{-1}\beta_T(x) \right)$
is the unique solution of \eqref{eq:cyclecond.1} in $(-\pi, \pi]^n$, 
where $\mathcal P(\phi)$ projects $\phi$ to $(-\pi, \pi]^n$.
\end{theorem}
\vspace{0.1in}
Theorem \ref{thm:anglecond} determines when the voltage magnitudes $v$ of a given 
$x$ can be assigned phase angles $\theta(x)$ so that the resulting $\tilde x := h^{-1}(x)$ is 
a power flow solution in $\mathbb{\tilde X}$.

\subsection{SOCP relaxation}

The set $\mathbb X_{nc}$ that contains the (equivalent) feasible set $\mathbb X$ of OPF
is still nonconvex because of the quadratic equalities in \eqref{eq:mdf.3}.  Relax them  to
inequalities:
\bq
v_j  \,  \ell_{jk} & \geq &  |S_{jk}|^2,  \qquad   (j, k)\in \tilde{E}
\label{eq:mdf.socp}
\eq
and define the set: 
 \bqn
{\mathbb X}^+ & \!\!\!\!\!\!\!\! := \!\!\!\!\!\!\!\! & 
	\{ x \in \mathbb R^{3(m+n+1)} \ \vert \ x \text{ satisfies } 
	 \eqref{eq:opfv}, \eqref{eq:opfs}, \eqref{eq:mdf.1}, \eqref{eq:mdf.2}, \eqref{eq:mdf.socp}
					  \}
\eqn
Clearly 
$\mathbb{\tilde X} \stackrel{h}{\equiv} \mathbb X \subseteq \mathbb X_{nc} \subseteq \mathbb X^+$;
see Figure \ref{fig:Xsets}.
Moreover $\mathbb X^+$ is a second-order cone in the rotated form.
\begin{figure}[htbp]
\centering
\includegraphics[width=0.3\textwidth]{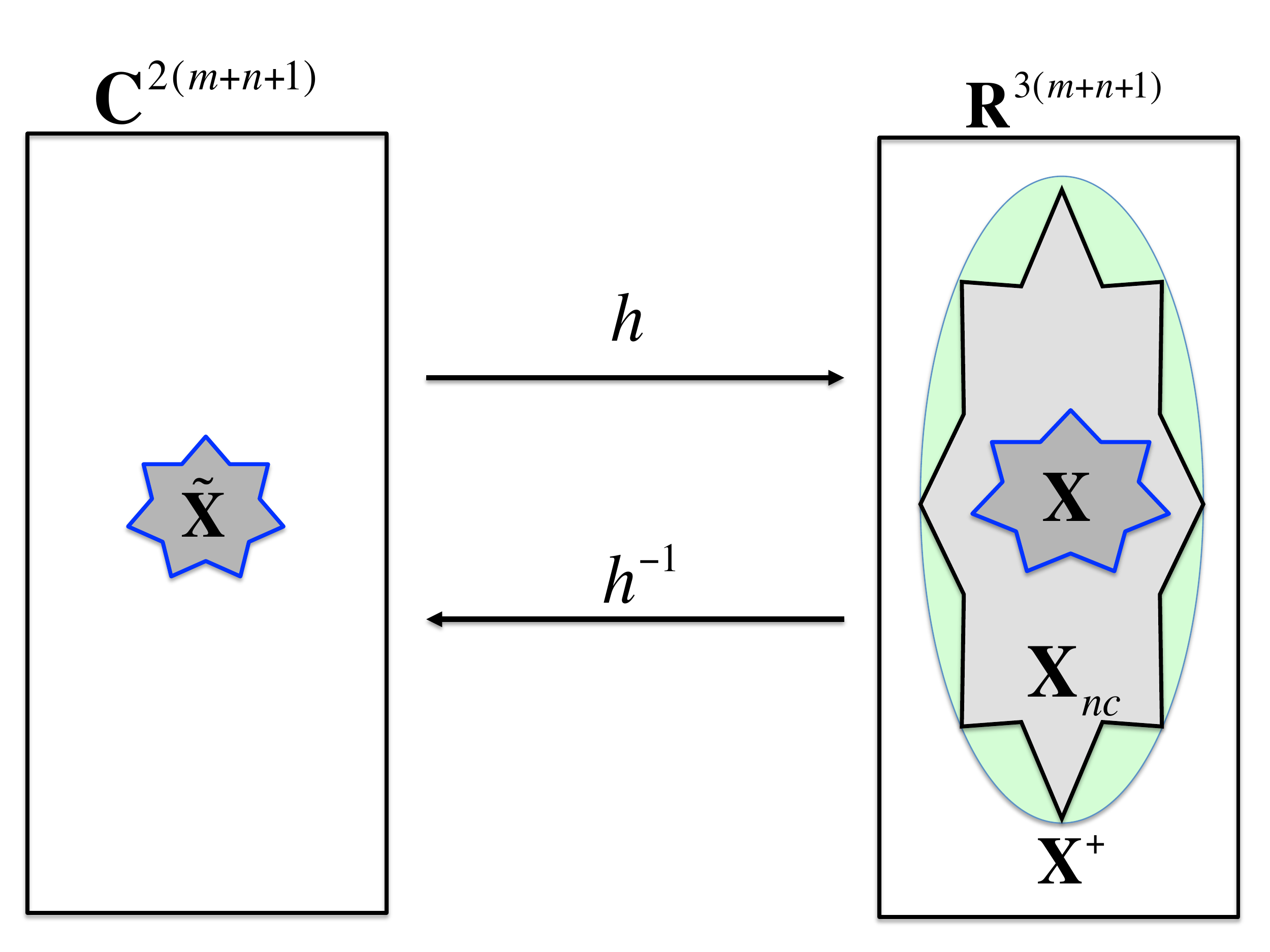}
\caption{Feasible sets $\mathbb{\tilde X}$ of OPF \eqref{eq:OPFbfm} in BFM, 
	its equivalent set $\mathbb X$ (defined by $h$)
	and its relaxations $\mathbb X_{nc}$ and $\mathbb X^+$.  If $\tilde G$ is a tree then
	$\mathbb X = \mathbb X_{nc}$.
}
\label{fig:Xsets}
\end{figure}

The three sets $\mathbb X$, $\mathbb X_{nc}$, $\mathbb X^+$
define the following problems:
\\
\noindent
\textbf{OPF:}
\bq
 \underset{ x } {\text{min}}   \ \  C( x) &  \text{ subject to } &  x \in \mathbb X
 \label{eq:bfmOPF-e}
\eq

\noindent
\textbf{OPF-nc:}
\bq
 \underset{ x} {\text{min}}   \ \  C( x) & \text{ subject to } &  x \in \mathbb X_{nc}
 \label{eq:bfmOPF-nc}
\eq

\noindent
\textbf{OPF-socp:}
\bq
\underset{ x } {\text{min}}   \ \  C( x) & \text{ subject to } &  x \in \mathbb X^+
\label{eq:bfmOPF-socp}
\eq

The next theorem follows from the results in \cite{Farivar-2013-BFM-TPS}
and implies that OPF \eqref{eq:OPFbfm} is equivalent to minimization 
over $\mathbb X$ and OPF-socp is its SOCP relaxation.
Moreover for radial networks voltage and current angles can be ignored and 
OPF \eqref{eq:OPFbfm} is equivalent to OPF-nc.
\begin{theorem} 
\label{thm:bfm.1}
$\mathbb{\tilde X} \equiv \mathbb X \subseteq \mathbb X_{nc} \subseteq \mathbb X^+$.
If $\tilde G$ is a tree then 
$\mathbb{\tilde X} \equiv \mathbb X = \mathbb X_{nc} \subseteq \mathbb X^+$.
\end{theorem}

Let $C^\text{opt}$ be the optimal cost of OPF \eqref{eq:OPFbfm} in the branch flow model.
Let $C^\text{opf}$, $C^\text{nc}$, $C^\text{socp}$ be the optimal costs of OPF \eqref{eq:bfmOPF-e}, 
OPF-nc \eqref{eq:bfmOPF-nc}, OPF-socp \eqref{eq:bfmOPF-socp}
respectively defined above.  Theorem \ref{thm:bfm.1} implies
\begin{corollary}
\label{coro:bfm.1}
$C^\text{opt} = C^\text{opf} \geq C^\text{nc} \geq C^\text{socp}$.   
If $\tilde G$ is a tree then 
$C^\text{opt} = C^\text{opf} = C^\text{nc} \geq C^\text{socp}$.
\end{corollary}

\begin{remark}
\label{remark:socp2}
\emph{SOCP relaxation.}
 Suppose one solves OPF-socp and obtains an optimal solution 
$x^\text{socp}:= (S, \ell, v, s) \in \mathbb X^+$.   For radial networks
if $x^\text{socp}$ attains equality in \eqref{eq:mdf.socp}
then $x^\text{socp} \in \mathbb X_{nc}$ and Theorem \ref{thm:bfm.1} implies 
that an optimal solution $\tilde x^\text{opt} := (S, I, V, s) \in \mathbb{\tilde X}$ of OPF 
\eqref{eq:OPFbfm} can be recovered from $x^\text{socp}$.
Indeed $\tilde x^\text{opt} = h^{-1}(x^\text{socp})$ where $h^{-1}$ is 
defined in \eqref{eq:defhinv}.
Alternatively one can use the angle recovery algorithms in \cite[Part I]{Farivar-2013-BFM-TPS}
to recover $\tilde x^\text{opt}$.
For mesh networks $x^\text{socp}$ needs to both attain equality in \eqref{eq:mdf.socp} and
satisfy the cycle condition \eqref{eq:cyclecond.1} in order for an optimal solution $\tilde x^\text{opt}$
to be recoverable. 
Our experience with various practical test networks suggests that
$x^\text{socp}$ usually attains equality in \eqref{eq:mdf.socp} but, for mesh networks, rarely satisfes \eqref{eq:cyclecond.1} 
\cite{Farivar2011-VAR-SGC, Farivar-2013-BFM-TPS, Gan-2014-BFMt-TAC, Bose-2014-BFMe-TAC}.  
Hence OPF-socp is effective for radial networks
 but not for mesh networks (in both BIM and BFM).
\end{remark}

\subsection{Equivalence}
\label{sec:eq}

Theorem \ref{thm:bfm.1} establishes a bijection between $\mathbb X$ and
the feasible set $\mathbb{\tilde X}$ of OPF  \eqref{eq:OPFbfm} in BFM. 
Theorem \ref{thm:FS} establishes a bijection between  $\mathbb W_G$ and
the feasible set $\mathbb V$ of OPF \eqref{eq:OPFbim} in BIM.
Theorem \ref{thm:bim=bfm} hence implies that 
$\mathbb X \equiv \mathbb{\tilde X} \equiv \mathbb V \equiv \mathbb W_G$.
Moreover their SOCP relaxations are equivalent in these two models
\cite{Bose-2012-BFMe-Allerton, Bose-2014-BFMe-TAC}.
Define the set of partial matrices defined on $G$ that are $2\times 2$ psd rank-1 but
do not satisfy the cycle condition \eqref{eq:cyclecond.2}:
\bqn
 \mathbb W_{nc} & := & \left\{\ W_{nc}  \ \vert \ W_G \text{ satisfies } \eqref{eq:opfW}, 
 			 W_G(j, k) \succeq 0, 
    	 	 \text{ rank } W_G(j,k) = 1 \text{ for all } (j,k)\in E \ \right\}
\eqn
Clearly $\mathbb W_G \subseteq \mathbb W_{nc} \subseteq \mathbb W_G^+$ in general
and $\mathbb W_G = \mathbb W_{nc} \subseteq \mathbb W_G^+$ for radial networks.
\begin{theorem}
\label{thm:eq}
$\mathbb X \equiv \mathbb W_G$,  $\mathbb X_{nc} \equiv \mathbb W_{nc}$ and 
$\mathbb X^+ \equiv \mathbb W_G^+$.
\end{theorem}

The bijection between $\mathbb X^+$ and $\mathbb W_G^+$ is a linear mapping defined as follows.
Let $\mathbb W_G \subseteq \mathbb C^{2m+n+1}$ denote the set of Hermitian 
partial matrices (including $[W_G]_{00} = v_0$ which is given).   Let ${x} := (S, \ell, v, s)$ 
denote vectors in $\mathbb R^{3(m+n+1)}$.
Define the linear mapping $g :\mathbb W_G^+ \rightarrow \mathbb X^+$ by 
${x}  = g(W_G)$ where 
\bqn
S_{jk} &:= &   y_{jk}^H \left( [W_G]_{jj} - [W_G]_{jk} \right),   \ \  j\rightarrow k 
\label{eq:g.1}
\\
\ell_{jk} & \!\!\!   := \!\!\!   &   |y_{jk}|^2 \left( [W_G]_{jj} + [W_G]_{kk} - [W_G]_{jk} - [W_G]_{kj} \right),
 \ j \rightarrow k
\label{eq:g.2}
\\
v_j & := &  [W_G]_{jj}, \ \   j\in N^+
\label{eq:g.3}
\\
s_j & := & \sum_{k: j\sim k} y_{jk}^H \left( [W_G]_{jj} - [W_G]_{jk} \right), \ \  j\in N^+
\eqn
Its inverse $g^{-1}: \mathbb X^+ \rightarrow \mathbb W_G^+$ is 
$W_G = g^{-1}(x)$ where
$[W_G]_{jj} := v_j$  for $j \in N^+$ and $[W_G]_{jk} := v_j - z_{jk}^H S_{jk} =: [W_G]_{kj}^H$
for $j\rightarrow k$.
The mapping $g$ (and its inverse $g^{-1}$) restricted to $\mathbb W_G$ $(\mathbb W_{nc})$ and 
$\mathbb X$ $(\mathbb X_{nc})$ define
the bijection between them.

%% file: 1-BFMt.tex
\section{BFM for radial networks}
\label{sec:BFMt}

Theorem \ref{thm:bfm.1} implies that for radial networks the model \eqref{eq:mdf} is exact.
This is because the reduced incident matrix $B$ in \eqref{eq:cyclecond.1} is $n\times n$ and 
invertible, so the cycle condition is always satisfied \cite[Theorem 4]{Farivar-2013-BFM-TPS}.
Hence a solution in $\mathbb X_{nc}$
can be mapped to a branch flow solution in $\mathbb{\tilde X}$ by the mapping $h^{-1}$
defined in \eqref{eq:defhinv}.
For radial networks this model has two advantages: (i) it has a recursive structure that simplifies
 computation, and (ii) it has a linear approximation that provides
simple bounds on branch powers $S_{jk}$ and voltage magnitudes $v_j$,
as we now show.

\subsection{Recursive equations and graph orientation}
\label{subsubsec:re}

The model \eqref{eq:mdf} holds for any graph orientation of $\tilde G$.  It has a recursive
structure when $\tilde G$ is a tree.  In that case different orientations have different boundary 
conditions that initialize the recursion and may be convenient for different applications.
Without loss of generality we take bus 0 as the root of the tree.  
We discuss two different orientations: one where every link points \emph{away from} bus 0 and the other where every link points \emph{towards} bus 0.
\footnote{An alternative model
is to use an \emph{un}directed graph and, for each link $(j,k)$, the variables $(S_{jk}, \ell_{jk})$ 
and $(S_{kj}, \ell_{kj})$ are defined for both directions, with the additional equations 
$S_{jk} + S_{kj} = z_{jk}\ell_{jk}$ and $\ell_{kj} = \ell_{jk}$.
}

\noindent
\emph{Case I:  Links point away from bus 0.}
Model \eqref{eq:mdf} reduces to:
\begin{subequations}
\bq
\!\!\!\!\!\!\!\!\!\!\!\!\!\!\!\!
\sum_{k: j\rightarrow k} S_{jk} & \!\!\! = \!\!\! & S_{ij} - z_{ij} \ell_{ij} + s_j, \quad j \in N^+
\label{eq:df.1}
\\
\!\!\!\!\!\!\!\!\!\!\!\!\!\!\!\!
v_j - v_k &  \!\!\! = \!\!\!  & 2\, \text{Re} \left(z_{jk}^H S_{jk} \right) - |z_{jk}|^2 \ell_{jk}, \ j\rightarrow k \in \tilde E
\label{eq:df.2}
\\
\!\!\!\!\!\!\!\!\!\!\!\!\!\!\!\!
v_j \ell_{jk} &  \!\!\! = \!\!\!  & |S_{jk}|^2, \quad j\rightarrow k \in \tilde E
\label{eq:df.3}
\eq
\label{eq:df.a}
\end{subequations}
where bus $i$ in \eqref{eq:df.1} denotes the unique parent of node $j$
(on the unique path from node 0 to node $j$), with the understanding that
if $j=0$ then $S_{i0} := 0$ and $\ell_{i0} := 0$.
Similarly when $j$ is a leaf node\footnote{A node $j$ is a \emph{leaf} node if
there exists no $i$ such that $i\rightarrow j \in \tilde E$.}
 all $S_{jk}=0$ in \eqref{eq:df.1}.
The model \eqref{eq:df.a} is called the \emph{DistFlow} equations
and first proposed in \cite{Baran1989a, Baran1989b}.

Its recursive structure is exploited in \cite{ChiangBaran1990} to analyze the power
flow solutions given an $(s_j, j\in N)$, as we now explain using the special case 
of a linear network with $n+1$ buses that represents a main feeder.
To simplify notation denote $\left(S_{j(j+1)}, \ell_{j(j+1)} \right)$ and $z_{j(j+1)}$ by 
$(S_j, \ell_j)$ and $z_j$ respectively.  Then the DistFlow equations \eqref{eq:df.a} reduce to ($v_0$ is given):
\begin{subequations}
\bq
 \!\!\!\!\!\!\!\!
S_{j+1} & \!\!\!\!\! =  \!\!\!\!\! & S_{j} - z_{j} \ell_{j} + s_{j+1}, \ \  j = 0, \dots, n-1
\label{eq:cb.1}
\\
 \!\!\!\!\!\!\!\!
v_{j+1}&  \!\!\!\!\! =  \!\!\!\!\! & v_j - 2 \text{Re} (z_{j}^H S_{j}) + |z_{j}|^2 \ell_{j}, \ \ j = 0, \dots, n-1
\label{eq:cb.2}
\\
 \!\!\!\!\!\!\!\!
v_j \ell_{j} &  \!\!\!\!\!  =  \!\!\!\!\! & |S_{j}|^2, \ \  j = 0, \dots, n-1
\label{eq:cb.3}
\\
 \!\!\!\!\!\!\!\!
S_0 &  \!\!\!\!\! =  \!\!\!\!\! & s_0, \quad S_n \ = \ 0
\label{eq:cb.4}
\eq
\label{eq:cb}
\end{subequations}
Let $x_j := (S_j, \ell_j, v_j)$, $j \in N^+$.  If $s_0$ were known then
one can start with $(v_0, s_0)$ and use the recursion \eqref{eq:cb.1}--\eqref{eq:cb.3} 
to compute $x_j$ in terms of $s_0 = S_0$, i.e., \eqref{eq:cb} can
be collapsed into functions of the scalar variable $s_0$ (recall that $(s_j, j \in N)$ are given):
\bq
x_j & = & f_j(s_0), \quad j \in N^+
\label{eq:lnDF.sol1}
\eq
Use the boundary condition \eqref{eq:cb.4}, $S_n  =  f_{n}(s_0) = 0$,
to solve for the scalar variable $s_0$.   The other variables $x_j$ can then be computed
from \eqref{eq:lnDF.sol1}.
This method can be extended to a general radial network with laterals \cite{ChiangBaran1990}.
See also \cite{ZimmermanChiang1995, Srinivas2000} for techniques for solving the nonlinear 
equations \eqref{eq:lnDF.sol1}, 
and \cite{Kersting2002, Shirmohammadi1988} for a different recursive approach called the 
\emph{forward/backward sweep} for radial networks.

\noindent
\emph{Case II: Links point towards bus 0.}
Model \eqref{eq:mdf} reduces to:
\begin{subequations}
\!\!\!\!\!\!\!\!\!\!\!\!\!
\bq
\hat{S}_{ji} & \!\!\! = \!\!\! & \sum_{k: k\rightarrow j} \left( \hat{S}_{kj} - z_{kj} \hat \ell_{kj} \right) + s_j,
 \quad j \in N^+
\label{eq:bdf.1}
\\
\!\!\!\!\!\!\!\!\!\!\!\!\!
\hat v_k - \hat v_j &  \!\!\! = \!\!\!  & 2\, \text{Re} \left(z_{kj}^H \hat S_{kj} \right) - |z_{kj}|^2 \hat\ell_{kj},
\ k\rightarrow j \in \tilde E
\label{eq:bdf.2}
\\
\!\!\!\!\!\!\!\!\!\!\!\!\!
\hat v_k \hat\ell_{kj} &  \!\!\! = \!\!\!  & |\hat S_{kj}|^2, 
\ \ k\rightarrow j \in \tilde E
\label{eq:bdf.3}
\eq
\label{eq:bdf}
\end{subequations}
where $i$ in \eqref{eq:bdf.1} denotes the  node on the unique path between 
node 0 and node $j$.  The boundary condition is defined by
$S_{ji}=0$ in \eqref{eq:bdf.1} when $j=0$ and $S_{kj}=0, \ell_{kj}=0$ in \eqref{eq:bdf.1} 
when $j$ is a leaf node.
An advantage of this orientation is illustrated in the next subsection in
proving a simple bound on $\hat v_j$.  
The proof establishes formally that there is a bijection between the solution
set of \eqref{eq:df.a} and that of \eqref{eq:bdf}:
\bqn
- {S}_{kj} & \leftrightarrow & \hat S_{jk} - z_{jk} {\hat \ell}_{jk}
\\
{\ell}_{kj} & \leftrightarrow & \hat \ell_{jk} 
\\
{v}_j  & \leftrightarrow & \hat v_j 
\eqn

\subsection{Linear approximation and bounds}
\label{subsubsec:sb}

By setting $\ell_{jk}=0$ in  \eqref{eq:df.a}
we obtain a linear approximation of the the branch flow model, with 
the graph orientation where all links point away from bus $0$:
\begin{subequations}
\bq
\sum_{k: j\rightarrow k} S^{\text{lin}}_{jk} & = & S^{\text{lin}}_{ij}  + s_j, \quad j \in N^+
\label{eq:ldf.1}
\\
v^{\text{lin}}_j - v^{\text{lin}}_k & = & 2\, \text{Re} \left(z_{jk}^H S^{\text{lin}}_{jk} \right), 
	\ \ j\rightarrow k \in E
\label{eq:ldf.2}
\eq
\label{eq:ldf}
\end{subequations}
where bus $i$ in \eqref{eq:ldf.1} denotes the unique parent of bus 
$j$.   The boundary condition is: 
$S_{i0}^\text{lin} := 0$  in \eqref{eq:ldf.1} when $j=0$,
and $S_{jk}^\text{lin} =0$ in \eqref{eq:ldf.1} when $j$ is a leaf node. 
This is called the \emph{simplified DistFlow equations} in \cite{Baran1989b, Baran1989c}.
It is a good approximation of \eqref{eq:df.a} because the loss $z_{jk}\ell_{jk}$
is typically much smaller than the branch power flow $S_{jk}$.

The next result provides simple bounds on $(S, v)$ in
terms of their linear approximations $(S^{\text{lin}}, v^{\text{lin}})$.
 Denote by $\mathbb T_j$ the subtree rooted at bus $j$, including $j$.
We write ``$k\in \mathbb T_j$'' to mean node $k$ of $\mathbb T_j$ and
``$(k, l)\in \mathbb T_j$'' to mean edge $(k,l)$ of $\mathbb T_j$.
Denote by $\mathbb P_k$ the set of links on the unique path from bus 0 to bus $k$.

\begin{lemma}
\label{lemma:ldf1} Fix any $v_0$ and $s \in \mathbb R^{2(n+1)}$.
Let $(S, \ell, v)$  and
$(S^{\text{lin}}, v^{\text{lin}})$ be solutions
of \eqref{eq:df.a} and \eqref{eq:ldf} respectively with the given $v_0$ and $s$. 
Then
\bee
\item[(1)] For $i\rightarrow j \in E$
	\bqn
	S^{\text{lin}}_{ij} & = & - \sum_{k \in \mathbb T_j} s_k
	\\
	S_{ij} & = & - \sum_{k\in \mathbb T_j} s_k\ + \left( z_{ij} \ell_{ij} + \sum_{(k,l)\in \mathbb T_j} z_{kl} \ell_{kl} \right)
	\eqn
\item[(2)] For $i \rightarrow j \in E$, $S_{ij} \geq S^{\text{lin}}_{ij}$ with equality if only if 
	$\ell_{ij}$ and all $\ell_{kl}$ in $\mathbb T_j$ are zero.

\item [(3)]For $j\in N^+$
	\bqn
	v^{\text{lin}}_{j} & = & v_0 \ - 
			 \sum_{(i,k) \in \mathbb P_j} 2\, \text{Re} \left( z_{ik}^H S_{ik}^{\text{lin}} \right)
	\\
	v_{j} & = & v_0 \ -  \sum_{(i,k) \in \mathbb P_j} \left( 2\, \text{Re} \left( z_{ik}^H S_{ik} \right) -
					|z_{ik}|^2 \ell_{ik} \right)
	\eqn
\item[(4)] For $j\in N^+$, $v_j \leq  v_j^{\text{lin}}$.
\eee
\end{lemma}
Lemma \ref{lemma:ldf1} says that the power flow $S_{ij}$ on line $(i,j)$ equals the total load 
$- \sum_{k \in \mathbb T_j} s_k$ in the subtree rooted at node $j$ plus the total line loss
in supplying these loads.  
The linear approximation $S_{ij}^{\text{lin}}$ neglects the line losses and underestimates 
the required power to supply these loads.

Lemma \ref{lemma:ldf1}(1)--(3) can be easily proved by recursing on 
\eqref{eq:df.1}--\eqref{eq:df.2} and \eqref{eq:ldf}.
Since $S_{ij} \geq S^{\text{lin}}_{ij}$ but $|z_{ik}|^2 \ell_{ik} \geq 0$,
a direct proof of Lemma \ref{lemma:ldf1}(4) is not obvious.
Instead, one can make use of Lemma \ref{lemma:ldf2} below and 
define a bijection between the solutions
$(S, \ell, v)$ of \eqref{eq:df.a} and the solutions
$(\hat S, \hat \ell, \hat v)$ of \eqref{eq:bdf} in which $v = \hat v$.
It can be checked that the solutions of \eqref{eq:ldf} and those of \eqref{eq:lbdf} are related
by $S^\text{lin} = - \hat S^\text{lin}$ and $v^\text{lin} = \hat v^\text{lin}$.
Then Lemma \ref{lemma:ldf2}(4)  implies Lemma \ref{lemma:ldf1}(4).

A linear approximation of \eqref{eq:bdf} is (setting $\hat\ell_{kj}=0$):
\begin{subequations}
\bq
\hat{S}_{ji}^{\text{lin}} & = & \sum_{k: k\rightarrow j}  \hat{S}_{kj}^{\text{lin}}  + s_j, \quad j \in N^+
\label{eq:lbdf.1}
\\
\hat v_k^{\text{lin}}  - \hat v_j^{\text{lin}}  & = & 
		2\, \text{Re}\! \left(z_{kj}^H \hat S_{kj}^{\text{lin}}  \right), \  k\rightarrow j \in \tilde E
\label{eq:lbdf.2}
\eq
\label{eq:lbdf}
\end{subequations}
\begin{lemma}
\label{lemma:ldf2} Fix any  $v_0$ and $s \in \mathbb R^{2(n+1)}$. 
Let $(\hat S, \hat \ell, \hat v)$ and $(\hat S^{\text{lin}},  \hat v^{\text{lin}})$ 
be solutions of \eqref{eq:bdf} and \eqref{eq:lbdf} respectively with
the given $v_0$ and $s$.   Then
\bee
\item[(1)] For all $j\rightarrow i \in E$
	\bqn
	\hat S^{\text{lin}}_{ji} & = & \sum_{k \in \mathbb T_j} s_k
	\\
	\hat S_{ji} & = &  \sum_{k\in \mathbb T_j} s_k\  - \sum_{(k,l)\in \mathbb T_j} z_{kl} \hat \ell_{kl} 
	\eqn
\item[(2)] For all $j \rightarrow i \in E$, $\hat S_{ji} \leq \hat S^{\text{lin}}_{ji}$ with equality if and only if all $\ell_{kl}$ 
	 in $\mathbb T_j$ are zero.
	 
\item[(3)] For $j\in N^+$
	\bqn
	\hat v^{\text{lin}}_{j} & = & v_0 \ +  \sum_{(i,k) \in \mathbb P_j} 
					2 \, \text{Re} \left( z_{ik}^H \hat S_{ik}^{\text{lin}} \right)
	\\
	\hat v_{j} & = & v_0 \ +  \sum_{(i,k) \in \mathbb P_j} \left( 2 \, \text{Re} \left( z_{ik}^H \hat S_{ik} \right)  
					- |z_{ik}|^2 \hat \ell_{ik} \right)
	\eqn

\item[(4)] For $j\in N^+$, $\hat v_j \leq \hat v_j^{\text{lin}}$.
\eee
\end{lemma}
Lemma \ref{lemma:ldf2} says that the branch power $\hat S_{ji}$ (towards bus 0) equals  the total
power injection $\sum_{k \in \mathbb T_j} s_k$ in the subtree rooted at bus $j$ minus the line losses 
in that subtree.  The linear approximation $\hat S_{ji}^{\text{lin}}$ neglects the line losses and hence
overestimates the branch power flow.
Lemma \ref{lemma:ldf2} can be easily proved by recursing on 
\eqref{eq:bdf.1}--\eqref{eq:bdf.2} and \eqref{eq:lbdf}.


\begin{remark}
\label{remark:bnd}
\emph{Bounds for SOCP relaxation.}
Lemmas \ref{lemma:ldf1} and \ref{lemma:ldf2} do not depend on 
the quadratic equalities \eqref{eq:df.3} and \eqref{eq:bdf.3} as long as $\ell_{jk}\geq 0$.
In particular the lemmas hold if the equalities have been relaxed to inequalities
$v_j \ell_{jk}  \geq  |S_{jk}|^2$.
These bounds are used in \cite{Gan-2014-BFMt-TAC} to prove a sufficient condition for exact 
SOCP relaxation for radial networks.
\end{remark}

\begin{remark}
\label{remark:LDFvsDC}
\emph{Linear approximations.}
For radial networks the linear approximations \eqref{eq:ldf} and \eqref{eq:lbdf} of BFM have two advantages over the (linear) DC approximation of BIM.
First they have a simple recursive structure that leads to simple
bounds on power flow quantities.
Second DC approximation assumes $r_{jk}=0$, fixes voltage
magnitudes, and ignores reactive power, whereas \eqref{eq:ldf} and \eqref{eq:lbdf} 
do not.  This is important for distribution systems where $r_{jk}$ are not negligible, 
voltages can fluctuate significantly and reactive powers are used to regulate them.
On the other hand  \eqref{eq:ldf} and \eqref{eq:lbdf} are applicable only for radial networks
whereas DC approximation applies to mesh networks as well.
See also \cite{Coffrin2012} for a more accurate linearization of BIM that
addresses the shortcomings of DC OPF.
\end{remark}

%% file: 1-conc.tex
\section{Conclusion}
\label{sec:conc1}

We have presented a bus injection model and a branch flow model, 
formulated several relaxations of OPF, and proved
their relations.   These results suggest a new approach to
solving OPF summarized in Figure \ref{fig:SolStrat}.
\begin{figure}[htbp]
\centering
\includegraphics[width=0.6\textwidth]{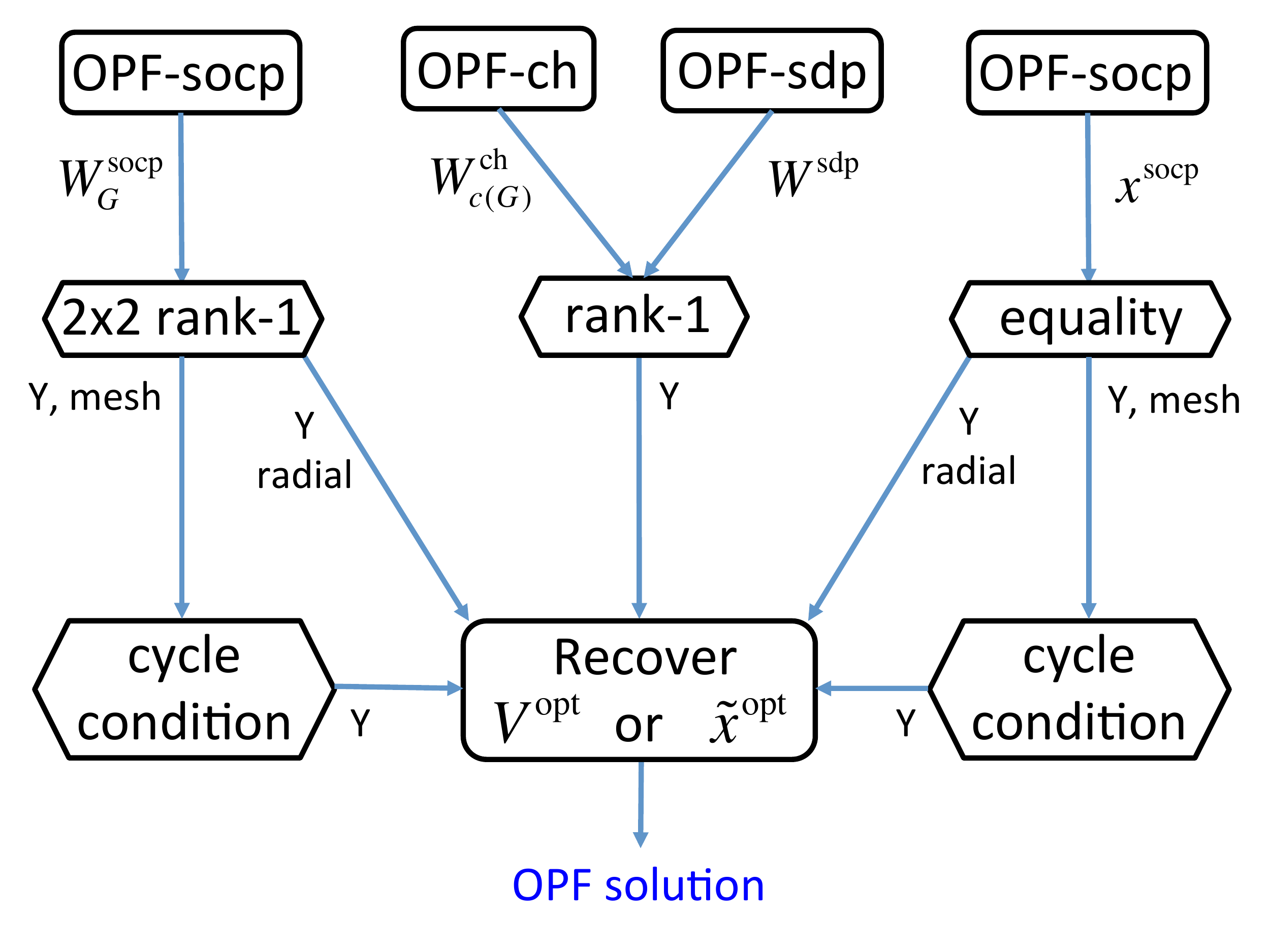}
\caption{Solving OPF through semidefinite relaxations.
}
\label{fig:SolStrat}
\end{figure}
For radial networks we recommend solving OPF-socp in either BIM or BFM
though there is preliminary evidence that BFM can be more stable
numerically.  For mesh networks we recommend solving OPF-ch for
small networks and OPF-socp followed by a heuristic search
for a feasible point for large networks.  Also see Remarks \ref{remark:socp1}
and \ref{remark:socp2}.

The key for this solution strategy is that the relaxations are exact
so that an optimal solution of the original OPF can be recovered.  
In Part II of this paper \cite{Low2014b} we summarize sufficient conditions that guarantee
exact relaxation.

%% file: 1-prelim.tex
\appendix[VIII: Mathematical preliminaries]
\label{app:prelims}


In this appendix we summarize  some basic concepts  in  optimization, 
matrix completion and chordal relaxation that we use in this two-part tutorial.
For notations see Section \ref{sec:intro1}.
More details can be found in, e.g., \cite{Boyd2004, wolkowicz00, Lobo1998,
Fukuda99exploitingsparsity, Nakata2003, zhang00, KimKojima2003, Grone1984}.

\subsection{QCQP, SDP, SOCP}
\label{subsec:opt}

Quadratic constrained quadratic program (QCQP) is the following problem:
\begin{subequations}
\bq
\min_{x\in \mathbb C^n} & & x^H C_0 x
\label{eq:defqcqp.1}
\\
\text{subject to}  & & x^H C_l x \leq b_l, \ \  l = 1, \dots, L
\label{eq:defqcqp.2}
\eq
\label{eq:defqcqp}
\end{subequations}
where $x\in \mathbb C^n$, for $l=0, \dots, L$, $C_l \in \mathbb S^n$ (so that $x^H C_l x$ 
are real), and $b_l \in \mathbb R$ are given.  
If $C_l$, $l = 0, \dots, L$, are positive semidefinite then \eqref{eq:defqcqp} is
a convex QCQP.  Otherwise it is generally nonconvex.

Any psd rank-1 matrix $X$ has a unique spectral decomposition $X = xx^H$.
Using $x^H C_l x = \text{tr } C_l xx^H =: \text{tr } C_l X$ we can  
rewrite a QCQP as the following equivalent problem where the optimization
is over Hermitian matrices:
\begin{subequations}
\bq
\min_{X \in \mathbb S^n} & & \text{tr } C_0 X
\label{eq:defqcqp.3}
\\
\text{subject to}  & &  \text{tr } C_l X \leq b_l, \ \  l = 1, \dots, L
\label{eq:defqcqp.4}
\\
	& & X \succeq 0, \ \  \text{rank } X = 1
\label{eq:defqcqp.5}
\eq
\label{eq:defqcqp.b}
\end{subequations}
While the objective function and the constraints in \eqref{eq:defqcqp} is quadratic
in $x$ they are linear in $X$ in \eqref{eq:defqcqp.3}--\eqref{eq:defqcqp.4}.  
The constraint $X \succeq 0$ in \eqref{eq:defqcqp.5} is convex 
($\mathbb S^n_+$ is a convex cone).  
The rank constraint in \eqref{eq:defqcqp.5} is the only nonconvex constraint.
Removing the rank constraint results in a semidefinite program (SDP):
\begin{subequations}
\bq
\min_{X \in \mathbb S^n} & & \text{tr } C_0 X
\label{eq:defsdp.1}
\\
\text{subject to}  & &  \text{tr } C_l X \leq b_l, \ \  l = 1, \dots, L
\label{eq:defsdp.2}
\\
	& & X \succeq 0
\label{eq:defsdp.3}
\eq
\label{eq:defsdp}
\end{subequations}
SDP is a convex program and can be efficiently computed.   
We call \eqref{eq:defsdp} an \emph{SDP relaxation}
of QCQP \eqref{eq:defqcqp} because the feasible set of \eqref{eq:defqcqp.b}
 is a subset of the feasible set of SDP \eqref{eq:defsdp}.
A  strategy for solving  QCQP \eqref{eq:defqcqp}
is  to solve SDP \eqref{eq:defsdp} for an optimal $X^\text{opt}$ and check its rank.
If rank $X^\text{opt} = 1$ then $X^\text{opt}$ is optimal for \eqref{eq:defqcqp.b} 
as well and an optimal solution $x^\text{opt}$ of QCQP \eqref{eq:defqcqp} can be 
recovered from $X^\text{opt}$ through spectral decomposition 
$X^\text{opt} = x^\text{opt} (x^\text{opt})^H$.  
If rank $X^\text{opt} > 1$ then, in general, no feasible solution of QCQP can be 
directly obtained from $X^\text{opt}$
but the optimal objective value of SDP provides a lower bound on that of QCQP.

To derive the Lagrangian dual of SDP \eqref{eq:defsdp}, form the Lagrangian,
for $y := (y_l, l=1, \dots, L) \geq 0$,
\bqn
L(X; y) & := & \text{tr}\, C_0 X \ + \ \sum_{l=1}^L y_l \left( \text{tr}\, C_l X - b_l \right)
 \ \, = \ \, \text{tr}\, \left( C_0 + \sum_l y_l C_l \right)\!\! X \ - \ b^T y
\eqn 
Then the primal problem \eqref{eq:defsdp} is equivalent to
$\min_{X \succeq 0} \, \max_{y \geq 0}\ L(X; y)$ and its dual is
$\max_{y \geq 0}\, \min_{X \succeq 0} \ L(X; y)$ (if we allow their objective values to be
$\pm\infty$).
Hence the dual objective function is
\bqn
\min_{X \succeq 0} \ L(X; y) & = & \left\{\begin{array}{lcl}
			- b^T y & & \text{if } C_0 + \sum_l y_l C_l \, \preceq \, 0
			\\
			-\infty & & \text{otherwise}
			\end{array}  \right.
\eqn
Hence the dual problem is:
\bqn
\min_{y \geq 0} \ b^T y
& \text{ subject to }  &  C_0 + \sum_{l=1}^L y_l C_l \, \preceq \, 0
\eqn
A pair $(X^\text{opt}, y^\text{opt})$ is a primal-dual optimal if and only if
\bee
\item \emph{Primal feasibility:} $X^\text{opt}\succeq 0$ and
	tr$\, C_l X^\text{opt} \, \leq \, b_l$, $l = 1, \dots, L$.
\item \emph{Dual feasibility:} $y^\text{opt}\geq 0$ and
	$C_0 \, + \, \sum_{l=1}^L \, y_l^\text{opt} \, C_l \ \preceq \ 0$.
\item \emph{Complementary slackness:} 
	tr$\, \left( C_0 \, + \, \sum_l \, y_l^\text{opt}\, C_l \right) X^\text{opt} \ = \ 0$.
\eee
\vspace{0.1in}

A special case of SDP is a second-order cone program (SOCP):
\begin{subequations}
\bq
\min_{x\in \mathbb C^n} & & c_0^H x
\label{eq:defsocp.a1}
\\
\text{subject to} & & \| C_l x + b_l \| \leq c_l^H x + d_l, \ \ l = 1, \dots, L
\label{eq:defsocp.a2}
\eq
\label{eq:defsocp.a}
\end{subequations}
where $c_0 \in \mathbb C^n$ defines the cost
and, for $l =1, \dots, L$, $C_l\in \mathbb C^{(n_l - 1) \times n}$, 
$b_l \in \mathbb C^{n_l - 1}$, $c_l\in \mathbb C^n$, and $d_l \in \mathbb R$
are given.   Here $c_l$, $l = 0, \dots, L$, are such that $c_l^H x$ are real
and $\|\cdot\|$ is the Euclidean norm, $\|u\| := \sqrt{u^H u}$.
The feasible set defined by \eqref{eq:defsocp.a2} is called a second-order cone
and is a convex set.
SOCP includes linear program and convex QCQP
 as special cases \cite{Lobo1998}.  Even though an SOCP can be formulated
as a standard SDP, solving an SOCP via SDP is generally much less efficient.
The number of iterations to reduce the duality gap to a constant fraction of itself
is bounded above by $O(\sqrt{L})$ for SOCP and by $O(\sqrt{\sum_l n_l})$
for SDP \cite{Lobo1998}.  Moreover each iteration is much faster for SCOP than
for SDP.    

For optimal power flow problems, we use SOCP in the following rotated form:
\bqn
\min_{x\in \mathbb C^n} & & c_0^H x
\\
\text{subject to} & & \| C_l x + b_l \|^2 \ \leq \ (c_l^H x + d_l)(\hat{c}_l^H x + \hat{d}_l),
			\qquad  l = 1, \dots, L
\eqn
This can be converted to the standard form  \eqref{eq:defsocp.a} via
the transformation: for any complex vector $u \in \mathbb C^l$, any real numbers
$a, b\in \mathbb R$, 
\bqn
\|u\|^2 \leq ab, \ a\geq 0, \ b\geq 0 & \Leftrightarrow & 
\left\| \begin{bmatrix}   2u \\ a-b   \end{bmatrix} \right\| \ \leq \ a+b
\eqn

In this paper we formulate optimal power flow (OPF) problems as
QCQPs and describe SDP and SOCP relaxations of OPF.  The third relaxation
we will discuss is chordal relaxation based on the notion of chordal extension
of a network graph.   We now review some basic concepts in graph theory, 
partial matrices and completions, and show that a chordal 
relaxation is indeed a semidefinite program.

\subsection{Graph, partial matrix and completion}
\label{subsec:pmc}

Consider a graph $G = (N, E)$ with $N := \{1, \dots, n\}$.  $G$ can either be
undirected or directed with an arbitrary orientation.
Two nodes $j$ and $k$ are \emph{adjacent} if $j\sim k \in E$.
A \emph{complete} graph is one where every pair of nodes is adjacent.
A subgraph of $G$ is a graph $F = (N', E')$ with $N'\subseteq N$ and $E' \subseteq E$.  
A \emph{clique} of $G$ is a complete  subgraph of $G$. A 
\emph{maximal clique} of $G$ is a clique that is not a subgraph of another 
clique of $G$. 

By a \emph{path} connecting nodes $j$ and
$k$ we mean either a set of \emph{distinct} nodes $(j, n_1, \dots, n_i, k)$ such that 
$(j \sim n_1), (n_1 \sim n_2), \dots, (n_i \sim k)$
are edges in $E$ or this set of edges, depending on the context.
A \emph{cycle} $(n_1, \dots, n_i)$ is a path such that 
$(n_1 \sim n_2), \dots, (n_i \sim n_1)$ are edges in $E$. 
By convention we exclude a pair of adjacent nodes $(j,k)$ as a cycle.
We will only consider connected graphs in which there is a path between
every pair of nodes.

A cycle in $G$ that has no chord (an edge connecting two nodes that 
are non-adjacent in the cycle) is called a \emph{minimal cycle}.
$G$ is  \emph{chordal} if all its minimal cycles are of length 3 (recall that
an edge $(j,k)$ is not considered a cycle).
A \emph{chordal extension} of  $G$ is a chordal graph 
on the same set of nodes as $G$ that contains $G$ as a subgraph. 
Every graph has a chordal extension; e.g. the complete graph on the
same set of nodes is a trivial chordal extension.

Fix a graph $G = (N, E)$ with $N:=\{1, \dots, n\}$ and $E \subseteq N\times N$.   
For our purposes here we assume $G$ is undirected
so that $(j,k) \in E$ if and only if $(k,j)\in E$.
A \emph{$G$-partial matrix} (or simply a \emph{partial
matrix} if $G$ is clear from the context) is a set of complex numbers:
\bqn
X_G \ := \  \left( [X_G]_{jj} \in \mathbb C, j\in N,\ [X_G]_{jk} \in \mathbb C, (j,k)\in E \right)
\eqn 
One can treat a partial matrix $X_G$ as entries of an $n \times n$ matrix $X$ 
whose entries $X_{jk}$ are unspecified if $(j,k)\not\in E$.  
See Figure \ref{fig:ChordExt}(a) below for an example.
Given a partial matrix $X_G$ we call an $n\times n$ matrix
$X$ a \emph{completion of $X_G$} if $X_{jj} = [X_G]_{jj}, j\in N$, and 
$X_{jk} = [X_G]_{jk}, (j,k)\in E$, i.e., $X$ \emph{agrees} with $X_G$ on $G$.\footnote{We 
abuse the $X_G$ notation: given $G$, $X_G$ is a partial matrix
defined on $G$, and given an $n\times n$ matrix $X$, $X_G$ is the submatrix
$(X_{jj}, j\in N, X_{jk}, (j,k)\in E)$ of $X$ defined by $G$.  
The meaning should be clear from the context.}

Consider any $n\times n$ matrix $X$.
Given any $k \leq n$ nodes $(n_1, n_2, \ldots, n_k)$ let $X(n_1, \ldots, n_k)$
denote the $k\times k$ principal submatrix of $X$ defined by:
\bqn
[X(n_1, \ldots, n_k)]_{ij} & := & X_{ij}, \ \ \ i, j \in \{n_1, \ldots, n_k\}
\eqn
Any maximal clique $q := (n_1, n_2, \ldots, n_k)$ of $G$ with $k$ nodes 
defines a (fully specified) $k\times k$ principal submatrix denoted by 
$X(q) := X(n_1, \ldots, n_k)$. 
In particular each edge $(i,j)\in E$ is a clique and defines a $2\times 2$ 
principal submatrix $X(i,j)$, which we use heavily in discussing optimal
power flow problems.  These notions are extended to partial matrices with 
$X$ replaced by $X_G$.

We extend the notions of Hermitian, psd,  rank-1, and trace to partial matrices as follows.
We say that a partial matrix $X_G$ is \emph{Hermitian}, denoted by $X_G = X_G^H$, if 
$[X_G]_{kj} = \left( [X_G]_{jk} \right)^H$.
An $n\times n$  matrix $X$ is psd if and only if all its principal submatrices (including
$X$ itself) is psd.  We extend the notion of psd to partial matrices using this
property, by saying that a partial matrix $X_G$ is psd if all its ``principal 
submatrices'' that are fully specified are psd.  Formally $X_G$ is psd, denoted by
$X_G \succeq 0$, if $X_G(q) \succeq 0$ for all maximal cliques $q$ of $G$.
Note that if $X_G$ is psd then it is Hermitian by definition.
Similarly we say that a partial matrix $X_G$ is \emph{rank-1}, 
denoted by rank $X_G = 1$, if $X_G(q)$ is rank-1 for all maximal cliques $q$ of $G$.
We say $W_G$ is \emph{$2\times 2$ psd} on $G$ if, for all $(j,k)\in E$,
the $2\times 2$ matrices $W_G(j,k)$ are psd, i.e., 
\bqn
[W_G]_{jj} \geq 0, \ \ [W_G]_{kk}\geq 0, \quad  
[W_G]_{jj}\, [W_G]_{kk} \ \geq \ \left|[W_G]_{jk}\right|^2
\eqn
We say $W_G$ is \emph{$2\times 2$ rank-1} on $G$ if, for all $(j,k)\in E$,
$W_G(j,k)$ are $2\times 2$ rank-1 matrices, i.e., they are not the zero matrices and 
\bqn
[W_G]_{jj}\, [W_G]_{kk} & = & \left| [W_G]_{jk} \right|^2
\eqn
Finally we say that an $n\times n$ matrix $C$ is \emph{defined on graph $G$}
if $C_{jk}=0$ if $(j,k)\not\in E$.
We extend the operation tr to partial matrices $X_G$: 
if $C$ and $X_G$ are defined on the same graph $G$ then
\bqn
\text{tr}\ C X_G & = & \sum_{j\in N} C_{jj} \, [X_G]_{jj} \ + \sum_{(j,k)\in E} C_{jk}\, [X_G]_{jk}
\eqn

Suppose the matrices $C_l$ in \eqref{eq:defsdp}, 
$l= 0,\dots, L$, are all defined on $G$, i.e., for all $l$, $[C_l]_{jk} = 0$ if 
$(j,k)\not\in E$.
Then given any $n\times n$ matrix $X$, tr $C_l X =$ tr $C_l X_G$ where
$X_G$ is the submatrix of $X$ defined by $G$.
Conversely, given a partial matrix $X_G$ that satisfies \eqref{eq:defsdp.2}, 
{\em any} completion $X$ of $X_G$ satisfies \eqref{eq:defsdp.2}.
Even though both the objective function \eqref{eq:defsdp.1} and the constraints
\eqref{eq:defsdp.2} depend only on the partial matrix $X_G$, the constraint 
$X\succeq 0$ in \eqref{eq:defsdp.3} depends also on entries not in $X_G$.
Indeed the number of complex variables in $X$ is $n^2$
while the number of complex variables in $X_G$ is only $n + 2|E|$, which is
much smaller than $n^2$ if $G$ is large but sparse.
Hence instead of solving for a full psd matrix $X$ directly as in SDP 
\eqref{eq:defsdp} we would like to compute a partial matrix $X_G$ that has 
a psd completion $X$ that satisfies \eqref{eq:defsdp.2}--\eqref{eq:defsdp.3}.
If the completion $X$ is rank-1 then it also solves the problem 
\eqref{eq:defqcqp.b} and hence yields a solution to
the original QCQP \eqref{eq:defqcqp} through spectral decomposition of $X$.
Theorem \ref{thm:rank1} provides an exact characterization of when this is
possible.

To solve the QCQP \eqref{eq:defqcqp}, Theorem \ref{thm:rank1} suggests the following
strategy that exploits the sparsity of graph $G$: instead of solving SDP
\eqref{eq:defsdp} for a psd matrix $X^\text{opt}\in \mathbb S^n_+$, 
solve for a psd partial matrix $X_F^\text{opt}$ defined on a chordal extension $F$ of $G$.
If the solution $X_F^\text{opt}$ turns out to be rank-1 as well then an optimal solution 
$x^\text{opt}$ of QCQP \eqref{eq:defqcqp} can be recovered from $X_F^\text{opt}$
(see Section \ref{subsec:sr}).

Two questions naturally arise in this approach: 
(i) How to formulate a 
semidefinite relaxation based on a given a chordal extension $F$ of $G$?
 (ii) How to choose a good chordal extension $F$ of $G$ so that the
	resulting relaxation can be solved efficiently?
We next illustrate the issues involved in these two questions through an
example.  See \cite{Fukuda99exploitingsparsity, Nakata2003} for more
details.

\subsection{Chordal relaxation}
\label{subsec:cr}

Fix a graph $G = (N, E)$.
Let $F = (N, E')$ be a chordal extension of $G$ with
 $E' \supseteq E$.   Let $q_1, \dots, q_K$ be the set of maximal 
cliques of $F$ and $X(q_k), k = 1,\dots, K$, be the set of principal
submatrices of $X$ defined on these cliques.  
 Consider the following problem 
where the optimization variable is the Hermitian partial matrix $W_F \in \mathbb C^{n+2|E'|}$ defined on the chordal extension $F$:
\begin{subequations}
\bq
\min_{X_F = X_F^H} & & \text{tr } C_0 X_G
\label{eq:defcr.1}
\\
\text{subject to}  & &  \text{tr } C_l X_G \leq b_l, \ \  l = 1, \dots, L
\label{eq:defcr.2}
\\
& & 		X_F(q_k) \succeq 0, \ \ k = 1, \dots, K
\label{eq:defcr.3}
\eq
\label{eq:defcr}
\end{subequations}
We call  this problem a \emph{chordal relaxation} of QCQP \eqref{eq:defqcqp}.
Recall that we assume $C_l$, $l = 0, \dots, L$, are all defined on $G$, i.e.,
$[C_l]_{jk} = 0$ if $(j,k)\not\in E$.
This implies that tr$\, C_l X = \ $tr$\, C_l X_F = \ $tr$\, C_l X_G$.
Then chordal relaxation \eqref{eq:defcr} is equivalent to SDP \eqref{eq:defsdp} 
in the sense that given 
any feasible solution $X_F$ of \eqref{eq:defcr}, there is a psd completion $X$ 
that is feasible for \eqref{eq:defsdp} and has the same cost, and vice versa.
This is a consequence of \cite[Theorem 7]{Grone1984} that says every psd partial 
matrix has a psd completion if and only if the underlying graph is chordal.
See also Theorem \ref{thm: bimR} and Corollary \ref{coro:bimR}.

The first step in constructing the chordal relaxation \eqref{eq:defcr} is to list
all the maximal cliques $q_k$.  Even though listing all maximal cliques of a
general graph is NP-hard it can be done efficiently for a chordal graph.
This is because a graph is chordal if and only if it has a perfect elimination
ordering \cite{Fulkerson1965} and computing this ordering takes linear
time in the number of nodes and edges \cite{Rose1976}.  Given a perfect
elimination ordering  all maximal cliques $q_k$ can be enumerated and
$X_F(q_k)$ constructed efficiently \cite{Fukuda99exploitingsparsity}.
For optimal power flow problems the computation depends only on the
topology of the power network, not on operational data, and therefore can 
be done offline.

We now show that  \eqref{eq:defcr} is 
indeed an SDP by converting it into the standard form \eqref{eq:defsdp} 
with the introduction of auxiliary variables, following the procedure described
in \cite{Fukuda99exploitingsparsity}.
This conversion also illustrates the difficulty in choosing a good 
chordal extension $F$ (see Remark \ref{remark:cr} below).

The (fully specified) matrices $X_F(q_k)$ in \eqref{eq:defcr.3} can be treated as
principal submatrices of an $n\times n$ matrix $X$.  
They may not however be integrated directly into a common $n\times n$ matrix variable
$X$ because different $X_F(q_k)$ may share entries.   We now explain the issue and
its resolution using the example in Figure 1.  They are the same in the  general case
with more cumbersome notations; see
\cite{Fukuda99exploitingsparsity, Nakata2003}.
\begin{figure*}[htbp]
\centering
\includegraphics[width=0.90\textwidth]{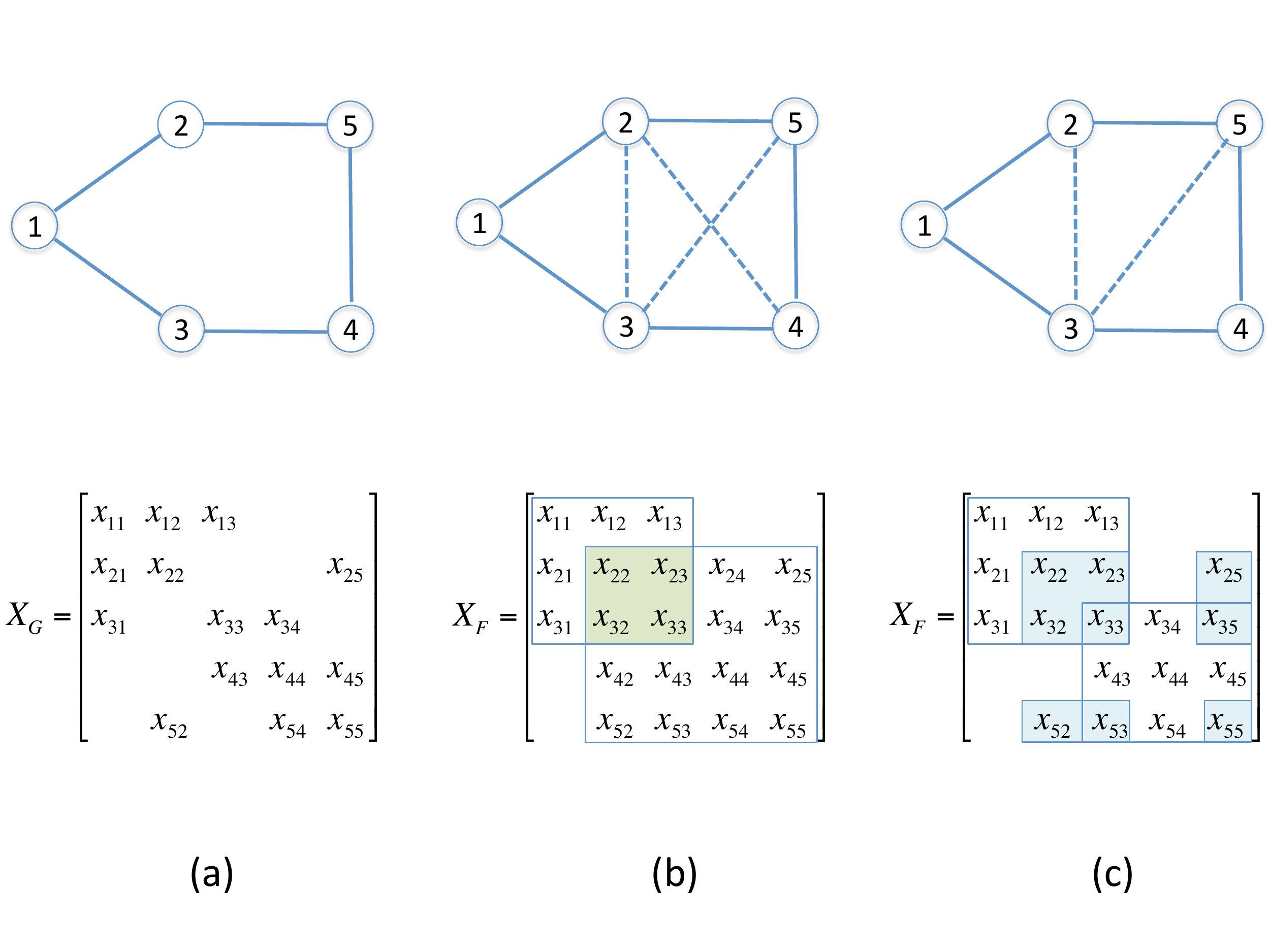}
\caption{Chordal extensions of $G$.  (a) Graph $G$ and the partial matrix $X_G$.
(b) A chordal extension $F$ and its $X_F$ that have 2 maximal cliques,
$q_1 := (1, 2, 3)$ and $q_2 := (2, 3, 4, 5)$.
These cliques share two nodes, 2 and 3.  The corresponding $X_F(q_1)$
and $X_F(q_2)$ are outlined in $X_F$ with the overlapping entries shaded in
green.  The chordal relaxation based on this $F$ requires 4 decoupling
variables $u_{jk}$.
 (c) Another chordal extension $F$ and its $X_F$ that have 3 maximal
cliques, outlined and shaded in blue in $X_F$.   The chordal relaxation based
on this $F$ requires 8 decoupling variables $u_{jk}$.
}
\label{fig:ChordExt}
\end{figure*}
Suppose we have chosen the chordal extension $F$ in Figure \ref{fig:ChordExt}(b)
with two overlapping cliques $q_1$ and $q_2$ as explained in the caption of the figure.
To decouple the two matrices $X_F(q_1)$ and $X_F(q_2)$, define the $3\times 3$ matrix
\bqn
X'(q_1) & := & \begin{bmatrix}
		x_{11}  &  x_{12}  &  x_{13}  \\
		x_{21}  &  u_{22}  &  u_{23}  \\
		x_{31}  &  u_{32}  &  u_{33}
		\end{bmatrix}
\eqn
where the decoupling variables $u_{jk}$ are constrained to be:
\bq
u_{jk} & = & x_{jk}  \quad \text{for } j, k = 2, 3
\label{eq:u}
\eq
The constraints \eqref{eq:defcr.3} are replaced by
\bq
X_F'(q_1) \succeq 0 & \text{and} & X_F(q_2) \succeq 0
\label{eq:qsdp}
\eq
Define the $7 \times 7$ block-diagonal matrix
\bqn
X' & := & \begin{bmatrix}
		X_F'(q_1)  &  0  \\
		0  &  X_F(q_2) 
		\end{bmatrix}
\eqn
Then  the chordal relaxation \eqref{eq:defcr} can be written in
the standard form \eqref{eq:defsdp} in terms of these
$7 \times 7$ block-diagonal Hermitian matrices:
\begin{subequations}
\bq
\min_{X' \in \mathbb S^7} & & \text{tr } C_0' X'
\\
\text{subject to}  & &  \text{tr } C_l' X' \leq b_l, \ \  l = 1, \dots, L
\\
		& & \text{tr } C_r' X' = 0, \ \ r = 1, 2, 3, 4
\label{eq:Cr'}
\\
		& & X' \succeq 0
\label{eq:X'}
\eq
\label{eq:cr}
\end{subequations}
for appropriate choices of $C_l'$, $l=0, \dots, L$.
The constraint $X' \succeq 0$ in \eqref{eq:X'} is equivalent to the  requirement
\eqref{eq:qsdp} on its submatrices 
and $C_r'$ in \eqref{eq:Cr'} is chosen to enforce the requirement \eqref{eq:u}.
Hence the chordal relaxation \eqref{eq:defcr} is indeed an SDP.

\begin{remark}
\label{remark:cr}
There are two conflicting factors in choosing a good chordal extension $F$.
First an $F$ that contains fewer number of maximal cliques $q$ 
generally involves larger cliques, leading
 to  larger submatrices $X_F(q)$; for example the complete graph $F$ has a single
maximal clique but the corresponding $X_F(q) = X$ has $n^2$ entries and the chordal
relaxation \eqref{eq:defcr} offers no 
computational advantage over solving (in fact it is exactly)
the original SDP \eqref{eq:defsdp}.
This argues for a chordal extension $F$ with smaller,  possibly more, maximal 
cliques $q$.
Second, however, having more maximal cliques $q$ tends to require more decoupling 
variables $u_{jk}$.
Every decoupling variable $u_{jk}$ introduces an extra equality constraint in \eqref{eq:Cr'},
thus increasing the required computational effort.
For instance the transformed problem based on the chordal extension in 
Figure \ref{fig:ChordExt}(b) involves 2 maximal cliques of sizes 3 and 4, and  
4 additional equality constraints in \eqref{eq:Cr'}.
 The transformed problem based on the chordal extension in 
 Figure \ref{fig:ChordExt}(c), on the other hand, requires 3 maximal cliques
 each of size 3, and 8 additional equality constraints.
\end{remark}

In summary even though the  ambient  dimension of the new variable $X'$ is generally
larger than that of the original  $n\times n$ matrix variable $X$ ($7 \times 7$ as
opposed to $5\times 5$ for the example in Figure \ref{fig:ChordExt}(b)), the chordal
relaxation \eqref{eq:defcr} can typically be solved much more efficiently than 
SDP \eqref{eq:defsdp} if $G$ is large and sparse;
for OPF examples, see \cite{MolzahnLesieutre2013, Bose-2014-BFMe-TAC}.   
Choosing a good chordal extension $F$ of $G$ is important but nontrivial. 
See  \cite{Fukuda99exploitingsparsity, Nakata2003} for methods to compute efficient
chordal extensions and sparse SDP solutions.

%% file: 1-AppendixProofs.tex
\appendix[IX: Proofs of main results]
\label{app:proofs1}


\subsection{Proof of Theorem \ref{thm:bim=bfm}: equivalence}

\begin{proof}
Fix any injections $s\in \mathbb C^{n+1}$.
It suffices to show that there is a bijection between the set of $V$ that satisfies
\eqref{eq:bim.1} and the set of $\tilde x := (S, I, \tilde V)$ that satisfies \eqref{eq:bfm}.
Here we use $V$ to denote a power flow solution in $\mathbb V(s)$ and $\tilde V$ to 
denote a component of a power flow solution $\tilde x$ in $\tilde{\mathbb X}(s)$.
We now exhibit a function $g$ that maps $V$ to $\tilde x$
and its inverse $g^{-1}$.

To construct $g$ fix any $V$ that satisfies \eqref{eq:bim.1}.
Define     
$\tilde V(V) := V$ and $I(V)$ by \eqref{eq:bfm.1}.  
Then define $S_{jk}(V)$ by \eqref{eq:bfm.2} (only) for $j\rightarrow k \in \tilde E$.  
This specifies $\tilde x := g(V) := (S(V), I(V), \tilde V(V))$.     We only need to show that
$\tilde x$ satisfies \eqref{eq:bfm.3}.  For each $i\rightarrow j \in \tilde E$ we have
by construction
\bqn
S_{ij}(V) - z_{ij} |I_{ij}(V)|^2 & = & 
y_{ij}^H\ V_i (V_i - V_j)^H - y_{ij}^H\ |V_i - V_j|^2
\ \, = \, \ - y_{ij}^H\ V_j (V_j - V_i)^H
\eqn
Hence for each $j\in N^+$
\bqn
\sum_{k: j\rightarrow k} S_{jk}(V)  \, -   
	\sum_{i: i\rightarrow j}  \left( S_{ij}(V) - z_{ij} |I_{ij}(V)|^2 \right)
& = & \sum_{k: j\rightarrow k} y_{jk}^H\ V_j (V_j - V_k)^H \, + 
	\sum_{i: i\rightarrow j}  y_{ij}^H\ V_j (V_j - V_i)^H
\\
& = & \sum_{k: j\sim k} y_{jk}^H\ V_j (V_j^H - V_k^H) \ \ = \ \ s_j(V)
\eqn
which is \eqref{eq:bfm.3}.

Conversely, given an $\tilde x := (S, I, \tilde V)$ that satisfies \eqref{eq:bfm}, 
define $V := g^{-1}(\tilde x) := \tilde V$.  
That $\tilde x$ satisfies \eqref{eq:bfm} implies that $g^{-1}$
as defined is indeed the inverse of $g$ above.  Moreover 
\eqref{eq:bfm.3} implies that $V$ satisfies \eqref{eq:bim.1}.
\end{proof}

\subsection{Proof of Theorem \ref{thm:rank1}: rank-1 characterization}
The proof  is from \cite{Bose-2014-BFMe-TAC}.

\begin{proof}
We will prove (1) $\Rightarrow$ (2) $\Rightarrow$ (3) $\Rightarrow$ (1).
If $W$ is psd rank-1 then all its principle submatrices are psd and of rank 1
(the submatrix cannot be of rank 0 because, by assumption,
$W_{jj}>0$ for all $j\in N^+$).
This implies that its submatrix $W_{c(G)}$ is psd and rank-1.
Hence (1) $\Rightarrow$ (2).

Fix a partial matrix $W_{c(G)}$ that is psd and rank-1 and consider its submatrix $W_G$.
Since each link $(j,k) \in E$ is a clique of $c(G)$ the $2\times 2$ principle submatrix $W_G(j,k)$ 
is psd and rank-1.  
Therefore to prove that (2) $\Rightarrow$ (3), it suffices to show that $W_G$
satisfies the cycle condition \eqref{eq:cyclecond.2}.   We now prove the following statement
by induction on $3 \leq k \leq n+1$: for all cycles $(j_1, \dots, j_k)$ of length $k$ in $c(G)$,
\bq
\sum_{i=1}^k\ \angle \left[W_G \right]_{j_i j_{i+1}} & = & 0   \mod 2\pi
\label{eq:mi}
\eq
where $j_{k+1} := j_1$.
For $k=3$, a cycle $(n_1, n_2, n_3)$ is a clique of $c(G)$ and therefore the following principle submatrix of $W_{c(G)}$:
\bqn
W_{c(G)}(n_1, n_2, n_3) & := & 
	\begin{bmatrix}
		[W_{c(G)}]_{n_1 n_1}  &  [W_{c(G)}]_{n_1 n_2}  & [W_{c(G)}]_{n_1 n_3} \\
		[W_{c(G)}]_{n_2 n_1}  &  [W_{c(G)}]_{n_2 n_2}  & [W_{c(G)}]_{n_2 n_3} \\
		[W_{c(G)}]_{n_3 n_1}  &  [W_{c(G)}]_{n_3 n_2}  & [W_{c(G)}]_{n_3 n_3} 
	\end{bmatrix}
\eqn
defined on the cycle is psd rank-1.   Hence $W_{c(G)}(n_1, n_2, n_3) = uu^H$ for some 
$u := (u_1, u_2, u_3) \in \mathbb C^3$.   Then
\bqn
\sum_{i=1}^3\ \angle \left[ W_G \right]_{j_i j_{i+1}} & = & 
\angle \left[ \left(u_1 u_2^H \right) \left( u_2 u_3^H \right) \left( u_3 u_1^H \right) \right]
\ \, = \, \ 0   \mod 2\pi
\eqn
Suppose \eqref{eq:mi} holds for all cycles in $c(G)$ of length up to $k>3$.  Consider now
a cycle $(j_1, \dots, j_{k+1})$ of length $k+1$ in $c(G)$.   Since  $c(G)$ is chordal
there is a chord, say, $(j_1, j_l)\in E$ for some $1<l<k+1$.   Since both cycles
$(j_1, \dots, j_l)$ and $(j_1, j_l, \dots, j_{k+1})$ satisfy \eqref{eq:mi} we have
\bqn
\sum_{i=1}^{l-1}\ \angle \left[ W_G \right]_{j_i j_{i+1}} \ + \ \angle \left[ W_G \right]_{j_l j_1} & = & 0   \mod 2\pi
\\
\angle \left[ W_G \right]_{j_1 j_l}  \ + \ \sum_{i=l}^{k+1}\ \angle \left[ W_G \right]_{j_i j_{i+1}} & = & 0   \mod 2\pi
\eqn
where $j_{k+2} := j_1$.
Since $W_G$ is Hermitian, adding the above equations yields
\bqn
\sum_{i=1}^{k+1} \ \angle \left[W_G \right]_{j_i j_{i+1}} & = & 0   \mod 2\pi
\eqn
proving \eqref{eq:mi} for $k+1$.  This completes the proof of (2) $\Rightarrow$ (3).

For (3) $\Rightarrow$ (1), fix any partial matrix $W_G$ that is $2\times 2$ psd rank-1 and
satisfies the cycle condition \eqref{eq:cyclecond.2}.   We now construct a psd rank-1 completion
$W$ of $W_G$, by constructing a vector $V\in \mathbb C^{n+1}$ such that $W = VV^H$.
Let 
\bqn
|V_j|  & := & \sqrt{ \left[ W_G \right]_{jj}}
\eqn
Without loss of generality let $\angle V_0 = 0^\circ$; for $j=1, \dots, n$, set
\bqn
\angle V_j  & := & - \sum_{(i,k)\in\mathbb P_j} \angle \left[ W_G \right]_{ik}
\eqn
where $\mathbb P_j$ is \emph{any} path from node 0 to node $j$.
This is well defined because $W_G$ satisfies the cycle condition \eqref{eq:cyclecond.2}.
It can be checked that $W = VV^H$ is indeed a psd rank-1 completion of $W_G$. 
This completes the proof.
\end{proof}

\subsection{Proof of Corollary \ref{coro:eq}: uniqueness of completion}

The proof is from \cite{Bose-2012-BFMe-Allerton}.
\begin{proof}
The proof of Theorem \ref{thm:rank1} shows that given a partial matrix 
$W_{c(G)} \in \mathbb W_{c(G)}$, the (unique) submatrix $W_G$ of $W_{c(G)}$
has a psd rank-1 completion $W\in \mathbb W$.   Therefore to prove the corollary it suffices
to prove that any partial matrix $W_G \in \mathbb W_G$ has a unique psd rank-1
completion $W\in \mathbb W$.   To this end fix a $W_G \in \mathbb W_G$ and suppose
there are two psd rank-1 completions $UU^H$ and $VV^H$ in $\mathbb W$.  
Clearly $|U_j| = |V_j| = \sqrt{\left[ W_G \right]_{jj}}$ for all $j\in N^+$; moreover 
$(\angle U_j, j\in N^+)$ and $(\angle V_j, j\in N^+)$ are solutions of
\bqn
B \theta & = & \beta  \qquad \text{ mod } 2\pi
\eqn
where $B$ is the $m\times n$ reduced incidence matrix defined in Section \ref{subsec:rBFMfs} 
and $\beta\in \mathbb R^m$ with $\beta_{jk} := \angle \left[ W_G \right]_{jk}$.
Since $G$ is connected, $m\geq n$ and hence the solution $\theta$ is unique.
Therefore $U=V$ and the psd rank-1 completion of $W_G$ is unique.
\end{proof}

\subsection{Proof of Theorem \ref{thm: bimR}: BIM feasible sets}

\begin{proof}
First $\mathbb V \sqsubseteq \mathbb W^+ \sqsubseteq \mathbb W_{c(G)}^+ \sqsubseteq \mathbb W_G^+$
follows from Theorem \ref{thm:FS} and the definitions of 
$\mathbb W^+$, $\mathbb W_{c(G)}^+$, $\mathbb W_G^+$ (recall that by assumption the cost
function $C$ depends on $V, W, W_{c(G)}$ only through the submatrix  $W_G$).
Since $c(G)$ is chordal, \cite[Theorem 7]{Grone1984} implies that every $W_{c(G)}$ in
$\mathbb W_{c(G)}^+$ has a psd completion $W$ in $\mathbb W^+$, i.e., 
$ \mathbb W_{c(G)}^+ \sqsubseteq \mathbb W^+$.   Hence $\mathbb W^+ \simeq \mathbb W_{c(G)}^+$.

Suppose $G$ is a tree and consider any chordal extension $c(G)$.   
We need to show that $\mathbb W_G^+ \sqsubseteq \mathbb W_{c(G)}^+$, i.e., given any 
$W_G \in \mathbb W_G^+$ there is a $W_{c(G)} \in \mathbb W_{c(G)}^+$ with the same cost.
Since $G$ is itself chordal, \cite[Theorem 7]{Grone1984} implies that $W_G$ has a 
psd completion $W$
in $\mathbb W^+$.  The submatrix $W_{c(G)}$ of $W$ defined on $c(G)$ is 
the desired partial
matrix in $\mathbb W_{c(G)}^+$ with the same cost.
This proves $\mathbb W_G^+ \sqsubseteq \mathbb W_{c(G)}^+$ and hence
$\mathbb W_G^+ \simeq \mathbb W_{c(G)}^+$ for radial networks.
\end{proof}

\subsection{Proof of Theorem \ref{thm:eqX}: BFM feasible sets}

That $\mathbb X \subseteq \mathbb X_{nc}$ follows from their definitions.  
Hence we only prove $\mathbb{\tilde X} \equiv \mathbb X$, following
\cite{Farivar-2013-BFM-TPS}.

Fix  any $x := (S, \ell, v, s)$ in ${\mathbb X_{nc}}$ and the corresponding 
$\beta(x)$ defined in (\ref{eq:defb}).  
Consider the cycle condition which is an equation in the variable $(\theta, k)$:
	\bq
	\label{eq:theta}
	B\theta & = & \beta(x) + 2\pi k
	\eq
where $k \in \mathbb N^m$ is an integer vector.  
Since $G$ is connected,  $m \geq n$ and rank$(B)=n$.  Hence, given any $k$,
there is at most one $\theta$ that solves \eqref{eq:theta}.  
Conversely, given a solution $\theta$, the corresponding $k$ is unique,
i.e., if the pairs $(\theta, k)$ and $(\theta, k')$ are both solutions of
\eqref{eq:theta}, then $k=k'$.   We therefore sometimes refer to a solution
$(\theta, k)$ simply as $\theta$ and write $k$ as
$k(\theta)$ when we want to emphasize its dependence on $\theta$.
Given any pair $(\theta, k)$ with $\theta \in (-\pi, \pi]^n$, define its 
{\em equivalence class} as
\bqn
\sigma(\theta, k) & := & \{ (\theta + 2\pi \alpha, k+ B\alpha) \ |\ \alpha \in \mathbb{N}^n \}
\eqn
Using the connectedness of $G$ and the definition of $B$, one can argue that $\alpha$ 
must be an integer vector for $k+B\alpha$ to be integral.
We say {\em $\sigma(\theta, k)$ is a solution of \eqref{eq:theta}} if every vector in
$\sigma(\theta, k)$ is a solution of \eqref{eq:theta}, and  
{\em $\sigma(\theta, k)$ is 
the unique solution of \eqref{eq:theta}} if it is the only equivalence class of solutions.
The lemma below implies that if a solution of \eqref{eq:theta} exists then it is unique.

We abuse (to simplify) notation and use $\theta$ to denote either an $n$-dimensional vector
$\theta := (\theta_j, j\in N)$ or an $(n+1)$-dimensional vector $\theta := (\theta_j, j\in N^+)$ 
with $\theta_0 := \angle V_0 := 0^\circ$, depending on the context.   
For each $\theta \in (-\pi, \pi]^{n+1}$, define 
the mapping $\tilde h_\theta(S, \ell, v, s) = (S, I, V, s)$ 
from $\mathbb{R}^{3(m+n+1)}$ to $\mathbb{C}^{2(m+n+1)}$  by:
\begin{subequations}
\bq
V_j & := & \sqrt{v_j} \ e^{\ii \theta_j},  \quad j \in N^+
\\
I_{jk} & := & \sqrt{\ell_{jk}}\ e^{\ii (\theta_j - \angle S_{jk})}, \  j\rightarrow k \in \tilde E
\eq
\label{eq:defhinv.2}
\end{subequations}
The proof of Theorem \ref{thm:eqX} relies on the following lemma that 
gives a necessary and sufficient condition on $\theta$ for 
$\tilde x := \tilde h_\theta(x)$ to be in $\tilde{\mathbb X}$.
\begin{lemma}
\label{lemma.1}
Fix any $x := (S, \ell, v, s)$ in ${\mathbb X_{nc}}$ and the corresponding 
$\beta(x)$ defined in (\ref{eq:defb}).  Then
\bee
\item $\tilde x := \tilde h_\theta(x)$ is in $\tilde{\mathbb X}$ if and only if
	$(\theta, k(\theta))$ solves \eqref{eq:theta}.
\item there is at most one $\sigma(\theta, k)$ with $\theta\in (-\pi, \pi]^n$, that is the
	unique solution of \eqref{eq:theta} when it exists.
\eee
\end{lemma}

\begin{proof}[Proof of Lemma \ref{lemma.1}]
To prove the first claim, 
suppose $(\theta, k)$ is a solution of \eqref{eq:cyclecond.1} for some $k = k(\theta)$.
We need to show that  (\ref{eq:mdf}), \eqref{eq:theta} together with
\eqref{eq:defhinv.2} imply (\ref{eq:bfm}).
Now \eqref{eq:mdf.1} is equivalent to \eqref{eq:bfm.3}.  Moreover
(\ref{eq:mdf.3}) and \eqref{eq:defhinv.2} imply \eqref{eq:bfm.2}.
To prove \eqref{eq:bfm.1}, substitute \eqref{eq:bfm.2} into \eqref{eq:theta} to get
\bqn
\theta_i - \theta_j & = & \angle \left( v_i - z_{ij}^H V_i I_{ij}^H \right) + 2\pi k_{ij}
\ \, = \, \  \angle \ V_i \left(V_i - z_{ij} I_{ij} \right)^H + 2\pi k_{ij}
\eqn
Hence
\bq
\angle V_j  & = & \theta_j \ = \ \angle \left(V_i - z_{ij} I_{ij}\right) - 2\pi k_{ij}
\label{eq:aVj}
\eq
From \eqref{eq:mdf.2}, we have
\bqn
|V_j|^2 & = & |V_i|^2 + |z_{ij}|^2 |I_{ij}|^2 - (z_{ij} S_{ij}^H + z_{ij}^H S_{ij})
\\
	& = & |V_i|^2 + |z_{ij}|^2 |I_{ij}|^2 - (z_{ij} V_i^H I_{ij} + z_{ij}^H V_i I_{ij}^H)
\\
	& = & |V_i - z_{ij} I_{ij}|^2
\eqn
where the second equality follows from \eqref{eq:bfm.2}.
This and (\ref{eq:aVj}) imply $V_j = V_i - z_{ij} I_{ij}$ which is (\ref{eq:bfm.1}).
This completes the proof that if $(\theta, k(\theta))$ is a solution of
\eqref{eq:theta} then $\tilde x := \tilde h_\theta (x)$ lies in $\tilde{\mathbb  X}$.

Conversely suppose $h_{\theta}(x) \in \mathbb X$.  
From \eqref{eq:bfm.1}--\eqref{eq:bfm.2}, we have
$V_i V_j^H = |V_i|^2 - z_{ij}^H S_{ij}$.   Then 
${\theta}_i - {\theta}_j = \beta_{ij} + 2\pi k_{ij}$
for some integer $k_{ij} = k_{ij}(\theta)$.   Hence $(\theta, k)$ solves \eqref{eq:theta}.

For the second claim, 
the discussion preceding the lemma shows that, given any $k\in \mathbb N^m$,
there is at most one $\theta$ that satisfies \eqref{eq:theta}.   If no such $\theta$ exists for any
$k\in\mathbb N^m$, then \eqref{eq:theta} has no solution $(\theta, k)$.  If \eqref{eq:theta} has
a solution $(\theta, k)$, then clearly $(\theta + 2\pi \alpha, k+ B\alpha)$ are also solutions
for all integer vectors $\alpha\in \mathbb{N}^n$.
Hence we can assume without loss of generality that $\theta\in (-\pi, \pi]^n$.  
We claim that $\sigma(\theta, k)$ is the unique solution of \eqref{eq:theta}.
Otherwise, there is an $(\tilde{\theta}, \tilde{k})\not\in \sigma(\theta, k)$ with
$B\tilde{\theta} = \beta + 2\pi \tilde{k}$.  
Then $B(\tilde{\theta} - \theta) = 2\pi (\tilde{k} - k)$, or 
\bqn
\tilde{k} & = & k + B\alpha 
	\qquad \text{ where } \ \ \alpha := \frac{1}{2\pi} (\tilde\theta - \theta)
\eqn
Note that both $k$ and $\tilde{k}$ are integer vectors in $\mathbb N^m$.
Using the connectedness of G and the definition of $B$, one can argue 
that $\alpha$ must be an integer vector in $\mathbb N^n$ for $k+B\alpha$
to be integral.   Then $\tilde \theta = \theta + 2\pi \alpha$ for some integer
vector $\alpha$. This means 
$(\tilde{\theta}, \tilde{k}) \in \sigma(\theta, k)$, a contradiction. 
\end{proof}
\vspace{0.1in}

\begin{proof}[Proof of Theorem \ref{thm:eqX}]
We now show that $\mathbb{\tilde X} \equiv \mathbb X$ by explicitly specifying
a bijection between these two sets.
Recall the mapping $h: \mathbb{\tilde X} \rightarrow \mathbb X$
defined by $h(S, I, V, s) = (S, \ell, v, s) =: x$ with $\ell_j := |I_j|^2$ and $v_j := |V_j|^2$.
Clearly it maps every $\tilde x \in \tilde{\mathbb X}$ to an $x$ that is in
$\mathbb X$.

To construct its inverse, consider the family of mappings $\tilde h_\theta$ 
defined in  \eqref{eq:defhinv.2} from 
$\mathbb{R}^{3(m+n+1)}$ to $\mathbb{C}^{2(m+n+1)}$, parameterized 
by \emph{every} $\theta \in \mathbb R^n$.  
Lemma \ref{lemma.1}(1) implies that, given any $x$ in $\mathbb X$, 
if we use a specific $\theta$, namely $\theta(x)$ that is a solution of 
\eqref{eq:theta}, then $\tilde h_{\theta(x)}(x)$ maps $x$  to an $\tilde x$ in 
$\tilde{\mathbb X}$.  
Moreover if we restrict $\theta(x)$ to $(-\pi, \pi]^{n+1}$ then 
$\tilde h_{\theta(x)}(x)$ is uniquely defined.  Consider then the overall mapping 
$\tilde h_{\theta(x)} (x)$ from $x\in \mathbb X$ to an $\tilde x \in\tilde{\mathbb X}$
with $\theta(x) \in (-\pi, \pi]^{n+1}$.
Note that for each $x\in \mathbb X$, this mapping selects a possibly different 
mapping from the family of mappings $\tilde h_\theta$, $\theta \in \mathbb R^{n+1}$
with $\theta_0(x) := \angle V_0 := 0^\circ$.
Clearly $h(\tilde h_{\theta(x)}(x)) = x$ and 
$\tilde h_{\theta( h(\tilde x))} (h(\tilde x)) = \tilde x$.
This means that $h^{-1}(\cdot) := \tilde h_{\theta(\cdot)} (\cdot)$ is an inverse of 
$h(\cdot)$.

This completes the proof of Theorem \ref{thm:eqX}.
\end{proof}

Notice that the inverse $h^{-1}(\cdot) := h_{\theta(\cdot)} (\cdot)$ is defined by the 
function $\theta(\cdot)$.  Even though $\theta$ is not unique (if $\theta$ 
defines an inverse of $h$, then $\theta+2\pi \alpha$ also defines a valid inverse), the
inverse $h^{-1}$ is unique because both $\theta$ and $\theta+2\pi \alpha$ map a point 
$x$ in $\mathbb X$ to the same point $\tilde x\in \mathbb{\tilde X}$ (see \eqref{eq:defhinv.2}).
Hence $\mathbb{\tilde X} \equiv \mathbb X$ is indeed well defined.
When we fix $\theta$ to be in $(-\pi, \pi]^n$, it fixes a unique $\theta$ that
defines $h^{-1}$ as well.

\subsection{Proof of Theorem \ref{thm:anglecond}: BFM cycle condition}
The proof is from \cite{Farivar-2013-BFM-TPS}.
\begin{proof}
Since $m \geq n$ and rank$(B)=n$, 
we can always find $n$ linearly independent rows of $B$ to form
a basis.    The choice of this basis corresponds to choosing a { spanning} tree of
$G$, which always exists since $G$ is connected \cite[Chapter 5]{Biggs-1993-agt}.
Assume without loss of generality that the first $n$ rows is such a basis
and partition $B$ and $\beta$ accordingly.
Then Lemma \ref{lemma.1} implies that $h_{\theta_*}(x) \in \mathbb X$ 
(in particular, $x$ satisfies the cycle condition \eqref{eq:cyclecond.1})
with $\theta_* \in (-\pi, \pi]^n$ if and only if $(\theta_*, k(\theta_*))$ is the unique solution
of
\bq
\begin{bmatrix}
B_T \\ B_{\perp}
\end{bmatrix}
\theta & = &
\begin{bmatrix}
\beta_T \\ \beta_{\perp}
\end{bmatrix}
+ 2\pi
\begin{bmatrix}
k_T \\ k_{\perp}
\end{bmatrix}
\label{eq:atb.2}
\eq
with $\theta \in (-\pi, \pi]^n$.
Since $T$ is a spanning tree, the $n\times n$ submatrix $B_T$ is invertible.  
Moreover \eqref{eq:atb.2} has a unique solution if and only if 
$B_{\perp} B_T^{-1} (\beta_T + 2\pi k_T) = \beta_{\perp} + 2\pi k_\perp$, 
or if and only if 
\bq
B_\perp B_T^{-1} \beta_T & = & \beta_\perp + 2\pi \hat{k}_\perp
\label{eq:cyclecond.1a}
\eq
for some integer vector $\hat{k}_\perp := k_\perp - B_\perp B_T^{-1}k_T$.   
(\eqref{eq:Binv} below implies that $\hat{k}_\perp$ is indeed an integer vector.)

We can assume without loss of generality that the orientation of the network
graph $\tilde G$ is such that all the links in $T$ are directed away from the root 
node $0$.  
Let $\mathbb P(i\leadsto j)$ denote the unique path from node $i$ to node $j$ in $T$;
in particular, $\mathbb P(0 \leadsto j)$ consists of links all with the same orientation 
as the path and $\mathbb P(j \leadsto 0)$ of links all with the opposite orientation.
Then it can be verified directly that 
\bq
\left[ B_T^{-1} \right]_{ie} & \!\!\! := \!\!\! & \begin{cases}
		-1  &  \text{ if link $e$ is in $\mathbb P(0 \leadsto i)$}
		\\
		0   &  \text{ otherwise}
		\end{cases}
\label{eq:Binv}
\eq 
Hence $B_T^{-1} \beta_T$ represents the (negative of the) sum of angle differences on the
path $\mathbb P(0\leadsto i)$ for each node $i\in T$:
\bqn
\left[ B_T^{-1} \beta_T \right]_i & = & \sum_e \left[B_T^{-1} \right]_{ie} \left[ \beta_T \right]_e
	\ = \ - \sum_{e \in \mathbb P(0\leadsto i)} \left[ \beta_T \right]_e
\eqn
Hence $B_\perp B_T^{-1} \beta_T$ is the sum of angle differences
from node $i$ to  node $j$ along the unique path in $T$, for every link
 $i\rightarrow j \in \tilde E\setminus \tilde E_T$
not in the tree $T$.
To see this, we have, for each link $e := i\rightarrow j \in \tilde E\setminus \tilde E_T$, 
\bqn
\left[ B_\perp B_T^{-1} \beta_T \right]_e & = & \left[ B_T^{-1} \beta_T \right]_i - \left[ B_T^{-1} \beta_T \right]_j
\ \, = \, \  \sum_{e' \in \mathbb P(0\leadsto j)} \left[ \beta_T \right]_{e'}  - 
	   \sum_{e' \in \mathbb P(0\leadsto i)} \left[ \beta_T \right]_{e'}
\eqn
Therefore \eqref{eq:cyclecond.1a} implies, for each link 
$i\rightarrow j \in \tilde E\setminus \tilde E_T$, 
\bq
\beta_{ij} \ - \ \left(
	   \sum_{e' \in \mathbb P(0\leadsto j)} \left[ \beta_T \right]_{e'}  - 
	   \sum_{e' \in \mathbb P(0\leadsto i)} \left[ \beta_T \right]_{e'}
	   \right)  
& = & - 2\pi \hat k_{\perp}
\label{eq:Binv2}
\eq
Consider the basis cycle $c(i,j)$ defined by such a link $i\rightarrow j$ 
outside the spanning tree $T$
\cite[Chapter 5]{Biggs-1993-agt}.   Using the definition of $\tilde \beta$ we have
\bqn
\sum_{e'\in c(i,j)} \tilde \beta_{e'} & = & 
		\tilde \beta_{ij} + \sum_{e' \in \mathbb P(j\leadsto 0)} \tilde\beta_{e'} 
		+ \sum_{e' \in \mathbb P(0\leadsto i)} \tilde\beta_{e'}
\\
& = & 	\beta_{ij} - \sum_{e' \in \mathbb P(0 \leadsto j)} \beta_{e'} 
		+ \sum_{e' \in \mathbb P(0\leadsto i)} \beta_{e'}
\ \, = \, \  -2\pi \left[ \hat k_{\perp} \right]_{ij}
\eqn
i.e., $\sum_{e'\in c(i,j)} \tilde \beta_{e'} = 0$ mod $2\pi$.
Since this holds for all basis cycles, it holds for all undirected cycles.
Therefore the cycle conditions \eqref{eq:cyclecond.1} and \eqref{eq:cyclecond.3}
are equivalent.

Finally consider the unique solution $(\theta_*, k(\theta_*))$ of \eqref{eq:theta}
with $\theta_* \in (-\pi, \pi]^n$.  
By \eqref{eq:atb.2} we have 
$\theta_* = B_T^{-1}\beta_T + 2\pi B_T^{-1} k_T(\theta_*)$.
The definition of $k(\theta_*)$ and the fact $\theta_* \in (-\pi, \pi]^n$
imply that $\theta_* = \mathcal{P}\left(B_T^{-1}\beta_T\right)$;
see the discussion preceding Lemma \ref{lemma.1}.
This completes the proof of Theorem \ref{thm:anglecond}.
\end{proof}

\subsection{Proof of Theorem \ref{thm:bfm.1}: radial networks}
The proof is from \cite{Farivar-2013-BFM-TPS}.
\begin{proof}
Theorem \ref{thm:eqX} and the definition of $\mathbb X^+$
imply
$\mathbb{\tilde X} \equiv \mathbb X \subseteq \mathbb X_{nc} \subseteq \mathbb X^+$.
If the network $\tilde G = T$ is a tree then
$B = B_T$ is $n\times n$ and invertible.  Hence every $x\in \mathbb X_{nc}$
satisfies the cycle condition \eqref{eq:cyclecond.1} and hence is in $\mathbb X$.
Therefore $\mathbb X = \mathbb X_{nc}$.
\end{proof}

\subsection{Proof of Theorem \ref{thm:eq}: equivalence}
The proof is from \cite{Bose-2012-BFMe-Allerton, Bose-2014-BFMe-TAC}.
\begin{proof}
Consider the linear mapping $g :\mathbb W_G^+ \rightarrow \mathbb X^+$ 
defined right after Theorem \ref{thm:eq} by ${x}  = g(W_G)$ where 
\begin{subequations}
\bq
S_{jk} &:= &   y_{jk}^H \left( [W_G]_{jj} - [W_G]_{jk} \right),   
			\qquad  j\rightarrow k \in \tilde E
\label{eq:g.1}
\\
\ell_{jk} &  :=  &   |y_{jk}|^2 \left( [W_G]_{jj} + [W_G]_{kk} - [W_G]_{jk} - [W_G]_{kj} \right),
			\qquad  j\rightarrow k \in \tilde E
\label{eq:g.2}
\\
v_j & := &  [W_G]_{jj}, \qquad   j\in N^+
\label{eq:g.3}
\\
s_j & := & \sum_{k: j\sim k} y_{jk}^H \left( [W_G]_{jj} - [W_G]_{jk} \right), \qquad j\in N^+
\label{eq:g.4}
\eq
\label{eq:g}
\end{subequations}
and the mapping $g^{-1}: \mathbb X^+ \rightarrow \mathbb W_G^+$ with
$W_G = g^{-1}(x)$ where
\begin{subequations}
\bq
[W_G]_{jj} & := & v_j,  \qquad j \in N^+
\label{eq:ginv.1}
\\
\left[ W_G \right]_{jk} & := & v_j - z_{jk}^H S_{jk} \ \, = \, \ \left[ W_G \right]_{kj}^H,
		\qquad j\rightarrow k \in \tilde E
\label{eq:ginv.2}
\eq
\label{eq:ginv}
\end{subequations}
We will prove that $g$ and $g^{-1}$ are indeed inverses of each other in three steps:
(1) $g$ maps every point $W_G \in \mathbb{W}_G^+$ to a point in $\mathbb X^+$;
(2) $g^{-1}$ maps every point $x \in \mathbb X^+$ to a point in $\mathbb{W}_G^+$;
and (3) $g(g^{-1}(x)) = x$ and $g^{-1}(g(W_G)) = W_G$.
This defines a bijection between $\mathbb W_G^+$ and $\mathbb X^+$
and establishes $\mathbb W_G^+ \equiv \mathbb X^+$.  
We will then prove the mappings $g$ ($g^{-1}$) restricted to $\mathbb W_G$ 
$(\mathbb W_{nc})$ and $\mathbb X$ $(\mathbb X_{nc})$ define
the bijection between these sets.

Recall the sets:
 \bqn
  \mathbb W_G^+ &  :=  &  \{W_G  \, \vert \, W_G \text{ satisfies } 
		\eqref{eq:opfW},   W_G(j,k) \succeq 0, \ (j,k)\in E \}
\\
{\mathbb X}^+ &  :=  & 
	\{ x \in \mathbb R^{3(m+n+1)} \ \vert \ x \text{ satisfies } 
	 \eqref{eq:opfv}, \eqref{eq:opfs}, \eqref{eq:mdf.1}, \eqref{eq:mdf.2}, \eqref{eq:mdf.socp}
					  \}
\eqn

\noindent
\emph{Step 1: $g(W_G) \in \mathbb X^+$.}
\eqref{eq:opfW} and \eqref{eq:g.3}--\eqref{eq:g.4} imply  \eqref{eq:opfv} and \eqref{eq:opfs}.
To prove \eqref{eq:mdf.1}, we have for $j \in N^+$
\bqn
&    & \sum_{i: i\rightarrow j} \left( S_{ij} - z_{ij} \ell_{ij} \right) + s_j
\\
& = &
\sum_{i: i\rightarrow j} \left( y_{ij}^H \left( [W_G]_{ii} - [W_G]_{ij} \right) 
	- y_{ij}^H \left( [W_G]_{ii} + [W_G]_{jj} - [W_G]_{ij} - [W_G]_{ji} \right) \right)
	\ + \ s_j
\\
& = &
\sum_{i: i\rightarrow j} \left( - y_{ij}^H \left( [W_G]_{jj}  - [W_G]_{ji} \right) \right)
	\ + \ \sum_{i: i\rightarrow j} y_{ji}^H \left( [W_G]_{jj} - [W_G]_{ji} \right)
	\ + \ \sum_{k: j\rightarrow k} y_{jk}^H \left( [W_G]_{jj} - [W_G]_{jk} \right)
\\
& = & \sum_{k: j\rightarrow k} S_{jk} 
\eqn
as desired.
To prove \eqref{eq:mdf.2}, we have for $j\rightarrow \in \tilde E$
\bqn
2\, \text{Re} \left(z_{jk}^H S_{jk} \right) - |z_{jk}|^2 \ell_{jk}
& = & 
2\, \text{Re} \left( [W_G]_{jj} - [W_G]_{jk} \right) 
	\ - \ \left( [W_G]_{jj} + [W_G]_{kk} - [W_G]_{jk} - [W_G]_{kj} \right)
\\
& = & \left( [W_G]_{jj} - [W_G]_{kk} \right) - [W_G]_{jk}^H + [W_G]_{kj} 
\\
& = & v_j  - v_k
\eqn
as desired.
To prove \eqref{eq:mdf.socp} for each $j\rightarrow \in \tilde E$, use 
$[W_G]_{jj} [W_G]_{jk} \geq | [W_G]_{jk} |^2$ to get
\bq
v_j \ell_{jk} & = & 
\left|y_{jk} \right|^2 \ [W_G]_{jj} \left( [W_G]_{jj} + [W_G]_{kk} - [W_G]_{jk} - [W_G]_{kj} \right)
\nonumber
\\
& \geq & 
\left|y_{jk} \right|^2  \left( [W_G]_{jj}^2 + \left| [W_G]_{jk} \right|^2 
	- [W_G]_{jj} [W_G]_{jk} - [W_G]_{jj} [W_G]_{jk}^H \right)
\label{eq:vl=S}
\\
& = & \left| S_{jk} \right|^2
\nonumber
\eq
as desired.
Hence $g$ maps every $W_G \in \mathbb W_G^+$ to a point in $\mathbb X^+$.

\noindent
\emph{Step 2: $g^{-1}(x) \in \mathbb W_G^+$.}
Clearly \eqref{eq:ginv.1} and  \eqref{eq:opfv} imply \eqref{eq:opfW.2}.
To prove \eqref{eq:opfW.1}, we have for each $j\in N^+$
\bqn
\sum_{k: (j,k)\in E} y_{jk}^H\, \left( [W_G]_{jj} - [W_G]_{jk} \right)
& = &
\sum_{i: i\rightarrow j} y_{ji}^H\, \left( [W_G]_{jj} - [W_G]_{ji} \right)  \ + \ 
\sum_{k: j\rightarrow k} y_{jk}^H\, \left( [W_G]_{jj} - [W_G]_{jk} \right)
\\
& = & 
\sum_{i: i\rightarrow j} y_{ij}^H\, \left( v_j - \left( v_i - z_{ij}^H S_{ij} \right)^H \right) \ + \ 
\sum_{k: j\rightarrow k} y_{jk}^H\, \left( v_j - \left( v_j - z_{jk}^H S_{jk} \right) \right)
\\
& = &
\sum_{k: j\rightarrow k} S_{jk} \ - \ 
\sum_{i: i\rightarrow j} y_{ij}^H\, \left( v_i - v_j - z_{ij} S_{ij}^H \right)
\\
& = &
\sum_{k: j\rightarrow k} S_{jk} \ - \ 
\sum_{i: i\rightarrow j} y_{ij}^H\, 
\left( 2\, \text{Re} (z_{ij}^H S_{ij}) - \left|z_{ij}\right|^2 \ell_{ij} - z_{ij} S_{ij}^H \right)
\eqn
where the second equality follows from \eqref{eq:ginv} and the last equality
from \eqref{eq:mdf.2}.  But
\bqn
\left( 2\, \text{Re} (z_{ij}^H S_{ij}) - z_{ij} S_{ij}^H \right)
& = & 
\left( z_{ij}^H S_{ij} + z_{ij} S_{ij}^H \right) - z_{ij} S_{ij}^H 
\ \ = \ \ z_{ij}^H S_{ij}
\eqn
and hence
\bqn
\sum_{k: (j,k)\in E} y_{jk}^H\, \left( [W_G]_{jj} - [W_G]_{jk} \right)
& = & 
\sum_{k: j\rightarrow k} S_{jk} \ - \ 
\sum_{i: i\rightarrow j}  \left( S_{ij} - z_{ij} \ell_{ij}  \right)  
\ \ = \ \ s_j
\eqn
This and \eqref{eq:opfs} imply \eqref{eq:opfW.1} as desired.
To prove $W_G(j,k) \succeq 0$ for each $(j,k)\in E$, we have
\bqn
[W_G]_{jj} [W_G]_{kk} - \left| [W_G]_{jk} \right|^2
& = & v_j v_k - \left| v_j - z_{jk}^H S_{jk} \right|^2
\\
& = & 
v_j v_k - \left( v_j^2 + \left|z_{jk}\right|^2 \left| S_{jk} \right|^2 -
	2 v_j\, \text{Re} \left( z_{jk}^H S_{jk} \right) \right)
\\
& = & 
v_j \left( v_k - v_j + 2 \, \text{Re} \left( z_{jk}^H S_{jk} \right) \right)
\ - \ \left|z_{jk}\right|^2 \left| S_{jk} \right|^2
\\
& = & 
 \left|z_{jk}\right|^2 \left( v_j \ell_{jk} - \left| S_{jk} \right|^2 \right)
\eqn
where the last equality follows from \eqref{eq:mdf.2}.
Therefore $W_G(j,k) \succeq 0$ follows from \eqref{eq:mdf.socp} as desired.
Hence $g^{-1}$ maps every point $x\in \mathbb X^+$ to a point in $\mathbb W_G^+$.

\noindent
\emph{Step 3: $g(g^{-1}(x)) = x$ and $g^{-1}(g(W_G)) = W_G$.}
The proof uses \eqref{eq:g}, \eqref{eq:ginv}, \eqref{eq:mdf} and follows similar
argument used in Steps 1 and 2, and is thus omitted.
This completes the proof that $g$ and $g^{-1}$ are indeed inverses of each other
and establishes $\mathbb W_G^+ \equiv \mathbb X^+$.

From \eqref{eq:vl=S}, we have 
\bqn
v_j \ell_{jk} = \left|S_{jk}\right|^2
& \text{ if and only if } &
[W_G]_{jj} [W_G]_{kk} = \left| [W_G]_{jk} \right|^2
\eqn
This implies that $g$ and $g^{-1}$ restricted to $\mathbb W_{nc}$ and
$\mathbb X_{nc}$ respectively are inverses of each other as well,
establishing $\mathbb W_{nc} \equiv \mathbb X_{nc}$.

Finally to prove $\mathbb W \equiv \mathbb X$ we need to show that 
the cycle conditions \eqref{eq:cyclecond.2} and  \eqref{eq:cyclecond.1}
are equivalent, i.e., 
$W_G \in \mathbb W_{nc}$ satisfies \eqref{eq:cyclecond.2} if and only if
$x := g(W_G) \in \mathbb X_{nc}$ satisfies \eqref{eq:cyclecond.1}.
Consider any cycle $c$.  Using \eqref{eq:ginv.2} and the definition
of $\tilde \beta(x)$ before Theorem \ref{thm:anglecond}, we have
\bqn
\sum_{(j,k)\in c} \angle W_{jk} & = & \sum_{(j,k)\in c} \angle \tilde \beta_{jk}
\eqn
Theorem \ref{thm:anglecond} therefore implies that the cycle conditions
\eqref{eq:cyclecond.2} and \eqref{eq:cyclecond.1} are equivalent under
$g$ and $g^{-1}$.   

This completes the proof of Theorem \ref{thm:eq}.
\end{proof}

\subsection{Proof of Lemma \ref{lemma:ldf1}: voltage bound}

The proofs of \ref{lemma:ldf1}(1)--(3) and Lemma \ref{lemma:ldf2} are
obvious and omitted.   The proof of Lemma
\ref{lemma:ldf1}(4) makes use of  Lemma \ref{lemma:ldf2} and is provided here.

\begin{proof}[Proof of Lemma \ref{lemma:ldf1}(3)]
Fix $v_0$ and an $s \in \mathbb R^{2(n+1)}$.  
Let $x := (S, \ell, v)$  and
$x^{\text{lin}} := (S^{\text{lin}}, \ell^{\text{lin}}, v^{\text{lin}})$ be solutions
of \eqref{eq:df.a} and \eqref{eq:ldf} respectively with the given $v_0$ and $s$,
when $\tilde G$ is oriented so that all links point away from node 0.   We will
prove $v\leq v^\text{lin}$, indirectly using Lemma \ref{lemma:ldf2}.

Specifically consider the model where the network is oriented so that all links 
point \emph{towards} node 0
and consider the solutions $\hat x := (\hat S, \hat \ell, \hat v)$ and 
$\hat x^\text{lin} := (\hat S^\text{lin}, \hat \ell^\text{lin}, \hat v^\text{lin})$ of \eqref{eq:bdf}
and \eqref{eq:lbdf} respectively with the given $v_0$ and $s$.
Then Lemma \ref{lemma:ldf2} implies $\hat v \leq \hat v^\text{lin}$.
We will prove that $v^\text{lin} = \hat v^\text{lin}$ and $v = \hat v$ and hence
$v\leq v^\text{lin}$ as desired.

It is easy to see from \eqref{eq:ldf} and \eqref{eq:lbdf} that 
$S_{jk}^\text{lin} = - \hat S_{kj}^\text{lin}$ and $v^\text{lin} = \hat v^\text{lin}$.

To show $v = \hat v$, define the following functions:
\begin{subequations}
\bq
{S}_{kj}(\hat S, \hat \ell, \hat v) & := & - (\hat S_{jk} - z_{jk} {\hat \ell}_{jk}),  \  j\rightarrow k \in E
\\
{\ell}_{kj}(\hat S, \hat \ell, \hat v) & := & \hat \ell_{jk},  \quad  j\rightarrow k \in E
\\
{v}_j(\hat S, \hat \ell, \hat v) & := & \hat v_j, \quad j \in N
\eq
\label{eq:Shat2S}
\end{subequations}
Note that $(S(\hat x), \ell(\hat x), v(\hat x))$ are \emph{functions} of $\hat x$ that is
a solution of \eqref{eq:lbdf}.\footnote{Since 
$\hat S_{jk}$ represents the sending-end complex power from bus $j$ to bus $k$, 
$- S_{kj}(\hat S, \hat \ell, \hat v)$ represents the received power by bus $k$ from bus $j$, net of 
the line loss $z_{jk}\ell_{jk}$.}
Substituting these functions into \eqref{eq:bdf} yields
\bq
- {S}_{ij}(\hat x) + z_{ji} \hat \ell_{ji} & \!\!\! = \!\!\! & - \sum_{k: k\rightarrow j} S_{jk}(\hat x) + s_j,
 \quad j \in N^+
\nonumber
\\
\!\!\!\!\!\!\!\!\!\!\!\!\!
v_k(\hat x) -  v_j(\hat x) &  \!\!\! = \!\!\!  & 
		2\, \text{Re} \left(z_{kj}^H \left( - S_{jk}(\hat x) + z_{kj} \hat\ell_{kj} \right) \right)
\nonumber
\\
\!\!\!\!\!\!\!\!\!\!\!\!\!
& &   \qquad\ 	 - |z_{kj}|^2 \hat\ell_{kj},    \ \ k\rightarrow j \in \tilde E
\nonumber
\\
\!\!\!\!\!\!\!\!\!\!\!\!\!
\ell_{jk}(\hat x)  v_k(\hat x) &  \!\!\! = \!\!\!  & \left| - S_{jk}(\hat x) + z_{kj}\hat \ell_{kj} \right|^2, 
\ \ k\rightarrow j \in \tilde E
\nonumber
\\
\label{eq:bdf.app1}
\eq
where as before $i$ in the above is the node on the unique path between 
node 0 and node $j$.   Substituting $\hat\ell_{kj} = \ell_{jk}(\hat x)$ and rearranging, 
these equations become
\begin{subequations}
\bq
\!\!\!\!\!\!\!\!\!\!\!\!\!
\sum_{k: k\rightarrow j} S_{jk}(\hat x) & \!\!\! = \!\!\! & {S}_{ij}(\hat x) - z_{ij} \ell_{ij}(\hat x) + s_j,
 \  j \in N^+
\\
\!\!\!\!\!\!\!\!\!\!\!\!\!
v_j(\hat x) -  v_k(\hat x) &  \!\!\! = \!\!\!  & 
		2\, \text{Re} \left(z_{jk}^H S_{jk}(\hat x) \right) - |z_{jk}|^2 \ell_{jk}(\hat x)
\nonumber
\\
& & \qquad\qquad\ \   k\rightarrow j \in \tilde E
\label{eq:bdf.app2}
\\
\!\!\!\!\!\!\!\!\!\!\!\!\!
\ell_{jk}(\hat x)  v_j(\hat x) &  \!\!\! = \!\!\!  & \left|S_{jk}(\hat x) \right|^2, 
\ \ k\rightarrow j \in \tilde E
\label{eq:bdf.app3}
\eq
\label{eq:bdf.app4}
\end{subequations}
To obtain \eqref{eq:bdf.app3} from \eqref{eq:bdf.app1}, note that the right-hand side 
of \eqref{eq:bdf.app1} is
\bqn
& &  \left| - S_{jk}(\hat x) + z_{jk} \ell_{jk}(\hat x) \right|^2 
\\
& = & 
\left| S_{jk}(\hat x) \right|^2 + \left( \left| z_{jk} \right|^2 \ell_{jk}(\hat x) - 
			2\, \text{Re} \left( z_{jk}^H  S_{jk}(\hat x) \right) \right) \ell_{jk}(\hat x)
\\
& = & \left| S_{jk}(\hat x) \right|^2 + \left(  v_k(\hat x) -  v_j(\hat x)   \right) \ell_{jk}(\hat x)
\eqn
where the second equality follows from \eqref{eq:bdf.app2}.
This together with \eqref{eq:bdf.app1} imply \eqref{eq:bdf.app3}.

Notice that \eqref{eq:bdf.app4} is identical to \eqref{eq:df.a} if we reverse the direction
of the graph $\tilde G$.  This means that the functions $(S(\hat x), \ell(\hat x), v(\hat x))$
are solutions of \eqref{eq:df.a}.   Moreover the functions \eqref{eq:Shat2S}
can be compactly written in vector form as:
\bqn
\begin{bmatrix}
S \\ \ell \\ v
\end{bmatrix}
(\hat S, \hat \ell, \hat v)
& = & 
\begin{bmatrix}
-I &  Z  &  0   \\
0 &  I   &  0   \\
0 &  0  &  I
\end{bmatrix}
\,
\begin{bmatrix}
\hat S \\ \hat \ell \\ \hat v
\end{bmatrix}
\eqn
where the $m\times m$ matrix $Z$ is diag$(z_{jk}, j\rightarrow k\in \tilde E)$.
Clearly this mapping is invertible.   This implies that there is a one-one correspondence 
between the solutions $x$ of \eqref{eq:df.a} and the solutions $\hat x$ of \eqref{eq:bdf}.
Moreover they have the same voltage magnitudes $v = \hat v$.

Hence Lemma \ref{lemma:ldf2}(4) implies that $v = \hat v \leq \hat v^\text{lin} = v^\text{lin}$
as desired.
\end{proof}